\documentclass[leqno]{amsart}
\usepackage{amsthm}
\usepackage{color}
\usepackage{amssymb}
\usepackage{amsmath,amsfonts,enumerate}
\usepackage{amsthm}
\usepackage{enumerate}

\numberwithin{equation}{section}

\newtheorem{thm}{Theorem}
\newtheorem{lem}[thm]{Lemma}

\newtheorem{rmk}[thm]{Remark}

\newtheorem{convention}[thm]{Convention}
\newtheorem{definition}[thm]{Definition}

\newtheorem{problem}[thm]{Problem}

\newtheorem{proposition}[thm]{Proposition}

\begin{document}

\title{The optimal range of the Calder\`{o}n operator and its applications
}

\author{F. Sukochev}
\address{School of Mathematics and Statistics, University of New South Wales, Kensington, NSW, 2052, Australia}
\email{f.sukochev@unsw.edu.au}
\author{K. Tulenov}
\address{
Al-Farabi Kazakh National University, 050040 Almaty, Kazakhstan;
Institute of Mathematics and Mathematical Modeling, 050010 Almaty, Kazakhstan.}
\email{tulenov@math.kz}
\author{D. Zanin}
\address{School of Mathematics and Statistics, University of New South Wales, Kensington, NSW, 2052, Australia}
\email{d.zanin@unsw.edu.au}


\date{}
\begin{abstract} We identify the optimal range of the Calder\`{o}n operator and that of the classical Hilbert transform in the class of symmetric quasi-Banach spaces. Further consequences of our approach concern the optimal range of the triangular truncation operator, operator Lipschitz functions and commutator estimates in ideals of compact operators.
\end{abstract}

\maketitle

\section{Introduction}

The classical Hilbert transform $\mathcal{H}$ (for measurable functions on $\mathbb R$) is given by the formula
$$(\mathcal{H}x)(t)=p.v.\frac{1}{\pi}\int_{\mathbb{R}}\frac{x(s)}{t-s}ds.$$
In \cite[Definition III.4.1, p. 126]{BSh}, it is stated that $\mathcal{H}f$ is defined for all locally integrable functions $f.$ This is certainly not the case. In fact, the maximal domain for $\mathcal{H}$ is the Lorentz space $\Lambda_{\log}(\mathbb{R})$ (see Remark \ref{dom of Hilbert trans}) associated with the function $\log(1+t),$ $t>0.$

Let $E$ and $F$ be symmetric quasi-Banach function spaces on $\mathbb{R}.$ In this paper, we are considering the problem of what is the least symmetric quasi-Banach space $F(\mathbb{R})$  such that $\mathcal{H}:E(\mathbb{R})\rightarrow F(\mathbb{R})$ is bounded for a fixed symmetric quasi-Banach space $E(\mathbb{R}).$
We shall be referring to the space $F(\mathbb{R})$ as the optimal range
space for the operator $\mathcal{H}$ restricted to the domain $E(\mathbb{R})\subseteq\Lambda_{\log}(\mathbb{R}).$

 If we restrict our attention to the subclass of symmetric {\it Banach spaces $E$ with Fatou norm} (that is, when the norm closed unit ball $B_E$ of $E$ is closed in $E$ with respect to almost everywhere convergence), then, in this (somewhat resticted) setting,  the problem reduces to a familiar problem settled by D. Boyd \cite{B} in 1967. Indeed, in this special case  \cite[Theorem 2.1]{B} asserts that $\mathcal{H}:E(\mathbb{R})\to F(\mathbb{R})$ if and only if $S:E(0,\infty)\to F(0,\infty),$ where the operator $S,$ known as the Calder\`{o}n operator, is defined by the formula
$$(Sx)(t)=\frac1t\int_0^tx(s)ds+\int_t^{\infty}\frac{x(s)}{s}ds,\quad x\in \Lambda_{\log}(0,\infty).$$
Effectively, the problem reduces to describing the optimal receptacle of the operator
$S.$

Consider the case when $E=L_p,$ $1<p<\infty.$ By Hardy's inequality, $S:L_p(0,\infty)\to L_p(0,\infty)$ and \cite[Theorem 2.1]{B} readily yields that $\mathcal{H}:L_p(\mathbb{R})\to L_p(\mathbb{R}).$ Now, if $E$ is a Banach interpolation space for the couple $(L_p,L_q),$ $1<p<q<\infty,$ then $\mathcal{H}:E(\mathbb{R})\to E(\mathbb{R})$ (this argument partly explains why Boyd's approach in \cite{B} has since become the mainstay of Interpolation Theory in symmetric function spaces, see \cite{BSh, KPS, LT}). A careful inspection of the proof in \cite{B} further yields that, in this case, $E(\mathbb{R})$ is the optimal receptacle for $\mathcal{H}$ restricted on $E(\mathbb{R}).$

In the case $p=1,$ Kolmogorov's classical result \cite{Kol} asserts that the operator $\mathcal{H}$ maps the symmetric Banach function space $L_1(\mathbb{R})$ to the symmetric {\it quasi-Banach} space $L_{1,\infty}(\mathbb{R})$ (see \cite[Theorem III.4.9 (b), p. 139]{BSh}). From this, one might guess that, for $E=L_1(\mathbb{R}),$ the optimal range of $\mathcal{H}$ is $L_{1,\infty}(\mathbb{R}),$ which strongly suggests that the natural setting for 
problem is that of symmetric quasi-Banach spaces.

Actually we study the optimal range of the operator-valued Hilbert transform $1\otimes\mathcal{H}.$ Let $(\mathcal{M},\tau)$ be a von Neumann algebra with a normal semifinite faithful trace $\tau.$ We shall discuss the optimal range of $1\otimes\mathcal{H}$ in the setting of symmetric quasi-Banach spaces on the von Neumann algebra tensor product $\mathcal{M}\bar{\otimes} L_{\infty}(\mathbb{R}),$ which are non-commutative counterparts of the classical symmetric function spaces (see Section 2 for precise definition).
Let $\mathcal{M}\bar{\otimes} L_{\infty}(\mathbb{R})$ be a von Neumann algebra tensor product with a normal semifinite faithful tensor trace. Define $1\otimes \mathcal{H}$ on $\Lambda_{\log}(\mathcal{M}\bar{\otimes}L_{\infty}(\mathbb{R}))$ (see Subsection 2.7 for detailed explanation) as a non-commutative Hilbert transform.

\begin{problem}\label{main problem} Let $\mathcal{M}$ be a semifinite von Neumann algebra equipped with a faithful normal semifinite trace $\tau.$ Given a symmetric quasi-Banach function space $E=E(\mathbb{R}),$ determine the least symmetric quasi-Banach function space $F=F(\mathbb{R})$ such that
 $1\otimes \mathcal{H}$ maps $\mathcal{E}(\mathcal{M}\otimes L_{\infty}(\mathbb{R}))$ to $\mathcal{F}(\mathcal{M}\otimes L_{\infty}(\mathbb{R})).$
\end{problem}


In the special case, when $\mathcal {M}=\mathbb{C}$, Problem \ref{main problem} has great lineage in Mathematical Analysis and Operator Theory as discussed above.  Addressing precisely this framework, one of our main results, Theorem \ref{quasi-banach opt range}, provides a complete description of the optimal range $F$ for a given symmetric quasi-Banach space $E,$ thereby complementing \cite[Theorem 2.1]{B}. Furthermore, Proposition \ref{optim weak l1} refines Kolmogorov's classical result \cite[Theorem III.4.9 (b), p. 139]{BSh}, by showing that the optimal range for the Hilbert transform on $L_1(\mathbb{R})$ is the space $(L_{1,\infty})^{0},$ the closure of all bounded functions in $L_{1,\infty}(\mathbb{R}).$ Finally, Theorem \ref{opt. range th H} resolves Problem \ref{main problem} in full generality.

The classical triangular truncation operator $T$ is defined in \cite{GK1,GK2} on integral operators (on the Hilbert space $L_2(\mathbb{R})$) by the formula
\begin{equation}\label{T-oper1}
(T(V)x)(t)=\int_{\mathbb{R}}K(t,s){\rm sgn}(t-s)x(s)ds,\quad x\in L_2(\mathbb{R}),
\end{equation}
where $K$ is the kernel of the integral operator $V,$
$$(Vx)(t)=\int_{\mathbb{R}}K(t,s)x(s)ds,\quad x\in L_2(\mathbb{R}).$$
We describe the exact domain of operator $T$ in the Remark \ref{T dom} below.
The class of Banach ideals $\mathcal{E}=\mathcal{E}(H)$ in $B(H)$ (equipped with Fatou norm) such that $T(\mathcal{E})\subseteq \mathcal{E}$ was discussed at great length
by I.C. Gohberg and M.G. Krein \cite{GK1,GK2}.
For an operator superficially similar to $T,$  J. Arazy (see \cite[Theorem 4.1]{Ar})
characterized all ideals $\mathcal{E}$ matching completely the above Boyd's results for the case
when the Boyd indices of $\mathcal{E}$ are not trivial.
However, even if $\mathcal{E}$ is an interpolation space for the couple $(\mathcal{L}_p,\mathcal{L}_q),$ $1<p<q<\infty,$
the techniques employed by Arazy, Gohberg and Krein \cite{Ar,GK1,GK2}, and in many other relevant papers, do not yield insight into whether $\mathcal{E}(H)$ is the optimal range space for the operator $T$ restricted to $\mathcal{E}(H)$. Thus, finding the optimal range space for the operator $T$
is of particular interest even in the simplest case, namely when the ideal $\mathcal{E}$ has non-trivial Boyd indices. A classical Macaev's result (see \cite{GK1,GK2,Mac}) and the Gohberg-Krein Theorem (see \cite[Theorem VII.5.1,  p. 345]{GK2}) suggest that the natural receptacle for the operator $T$ restricted to the trace class $\mathcal{L}_1(H)$ is the quasi-Banach ideal $\mathcal{L}_{1,\infty}(H)$ and, as is the case for Problem \ref{main problem}, this motivates our consideration of quasi-Banach ideals.

\begin{problem}\label{second problem} Given a symmetric quasi-Banach sequence space $E=E(\mathbb{Z}),$ determine the least symmetric quasi-Banach sequence space $F=F(\mathbb{Z})$ such that
$T:\mathcal{E}(H)\rightarrow \mathcal{F}(H).$
\end{problem}

Our main result concerning Problem 2 is presented in Theorem \ref{opt. range th T}. It matches completely the commutative results cited above and substantially extends earlier results from \cite[Theorem III.2.1, 2.2, 4.1, and 5.3, pp. 83-116]{GK2} (see also \cite[Corollary III. 7.1, p. 130]{GK2}).
The new components of our proof, which enable us to obtain optimal range results in Theorem \ref{opt. range th T}, are provided by Theorems \ref{first main theorem} and \ref{T theorem}, where the latter may be viewed as a non-commutative extension of Boyd's result given in \cite[Proposition III.4.10, p. 140]{BSh}.

The most notable achievement of this paper is our approach, which allows the simultaneous treatment of Problems 1 and 2 from the perspective of non-commutative symmetric spaces. The expert reader can skip to Sections \ref{statement section} and \ref{statement proof} for the statement of our main result Theorem \ref{first main theorem} and its proof. Following this, Section 5 is devoted to producing a lower estimate for the triangular truncation operator $T.$

In Section 6 we find the optimal range for the Calder\'{o}n operator $S.$ In Section 7 we fully resolve Problems \ref{main problem} and \ref{second problem} by explicitly formulating the optimal symmetric quasi-Banach range for the operators $1\otimes\mathcal{H}$ and $T$. Finally, our approach enables important applications to Double Operator Integrals associated with Lipschitz functions and commutator estimates, which we address in Section 8. We are able to deal with the most general assumptions on a quasi-Banach symmetric space $E$ omitting the assumption that it has a Fatou norm used in \cite{Ar, BSh,B, GK1, GK2} and many other sources.

Our final comment is to, on the one hand, connect our theme with the studies of UMD-spaces in the setting of symmetric operator spaces and, on the other hand, explain the profound difference with the latter. Recall that one of the equivalent characterizations of the so-called UMD-property of a Banach space $X$ is that \cite{Bour1, Bur} the Hilbert kernel for the real line
$\mathbb{R}$  defines a bounded convolution operator on the corresponding Bochner $L_{p}$-spaces ($1<p<\infty$) of $X$-valued functions.
In particular, Theorem 4.1 due to Zsido \cite{Zs} implies almost immediately that the spaces
$\mathcal{L}_p(\mathcal{M}, \tau), 1<p<\infty$, have the
UMD-property (see details in \cite[Corollary 4.2]{DDPS} and also \cite{BGM1}).  Equivalently, we see that $1\otimes \mathcal{H}$ maps $\mathcal{L}_p(\mathcal{M}\bar{\otimes }L_{\infty}(\mathbb{R}))$ to $\mathcal{L}_p(\mathcal{M}\bar{\otimes} L_{\infty}(\mathbb{R}))$ whenever $1<p<\infty$ and this connection with Problem \ref{main problem} has been employed in the studies of various versions of the non-commutative Hilbert Transform for which we refer to \cite{DDdePS, DDPS, Ran1, Ran2}. However, in the case when a symmetric space $E$ has trivial Boyd indices, it is not a UMD-space and therefore all the techniques from above cited papers are not applicable in the present setting.

\section{Preliminaries}

\subsection{Singular value functions}\label{s number section}

Let $(I,m)$ denote the measure space $I = (0,\infty),\mathbb{R}$ (resp. $I=\mathbb{Z}_+:=\mathbb{Z}_{\geq 0},\mathbb{Z}$), where $\mathbb{R}$ is the set of real (resp. $\mathbb{Z}$ the set of integer) numbers, equipped with Lebesgue measure (resp. counting measure) $m.$
 Let $L(I,m)$ be the space of all measurable real-valued functions (resp. sequences) on $I$ equipped with Lebesgue measure (resp. counting measure) $m$ i.e. functions which coincide almost everywhere are considered identical. Define $S(I,m)$ to be the subset of $L(I,m)$ which consists of all functions (resp. sequences) $x$ such that $m(\{t : |x(t)| > s\})$ is finite for some $s > 0.$

For $x\in S(I)$ (where $I = (0,\infty)$ or $\mathbb{R}$), we denote by $\mu(x)$ the decreasing rearrangement of the function $|x|.$ That is,
$$\mu(t,x)=\inf\{s\geq0:\ m(\{|x|>s\})\leq t\},\quad t>0.$$
 On the other hand, if $I = \mathbb{Z}_+, \mathbb{Z}$, and $m$ is the counting measure, then $S(I) = \ell_\infty (I)$, where $\ell_\infty (I)$ denotes the space of all bounded sequences on $I$. In this case, for a sequence $x=\{x_n\}_{n\geq0}$ in $\ell_{\infty}(\mathbb{Z}_{+})$ (resp. $\ell_{\infty}(\mathbb{Z})$), we denote by $\mu(x)$ the sequence $|x|=\{|x_{n}|\}_{n\geq0}$ rearranged to be in decreasing order.

We say that $y$ is submajorized by $x$ in the sense of Hardy--Littlewood--P\'{o}lya (written $y\prec\prec x$) if
$$\int_0^t\mu(s,y)ds\leq\int_0^t\mu(s,x)ds,\quad t\geq0$$
$$\Big(\text{or}\,\, \sum_{k=0}^{n}\mu(k,y)\leq\sum_{k=0}^{n}\mu(k,x),\quad n\geq0\Big).$$

Let $\mathcal{M}$ be a semifinite von Neumann algebra on a separable Hilbert space $H$ equipped with a faithful normal semifinite trace $\tau.$
A closed and densely defined operator $A$ affiliated with $\mathcal{M}$ is called $\tau$-measurable if $\tau(E_{|A|}(s,\infty))<\infty$ for sufficiently large $s.$ We denote the set of all $\tau$-measurable operators by
$S(\mathcal{M},\tau).$ Let $Proj(\mathcal{M})$ denote the lattice of all projections in $\mathcal{M}.$ For every $A\in S(\mathcal{M},\tau),$ we define its singular value function $\mu(A)$ by setting
$$\mu(t,A)=\inf\{\|A(1-P)\|_{L_{\infty}(\mathcal{M})}:P\in Proj(\mathcal{M}),\quad \tau(P)\leq t\}, \quad t>0.$$
Equivalently, for positive self-adjoint operators $A\in S(\mathcal{M},\tau),$ we have
$$n_A(s)=\tau(E_A(s,\infty)),\quad \mu(t,A)=\inf\{s: n_A(s)<t\}, \quad t>0.$$
For more details on generalised singular value functions, we refer the reader to \cite{FK} and \cite{LSZ}.
We have for $A,B \in S(\mathcal{M},\tau)$ (see for instance \cite[Corollary 2.3.16 (a)]{LSZ})
\begin{equation}\label{triangle svf}
\mu(t+s,A+B)\leq\mu(t,A)+\mu(s,B),\quad t,s>0.
\end{equation}

If $A,B\in S(\mathcal{M},\tau),$ then we say that $B$ is submajorized by $A$ (in the sense of Hardy--Littlewood--P\'{o}lya), denoted by $\mu(B)\prec\prec\mu(A),$ if
$$\int_0^t\mu(s,B)ds\leq\int_0^t\mu(s,A)ds,\quad t\geq0.$$
If $\mathcal{M}=B(H)$ and $\tau$ is the standard trace $Tr,$ then it is not difficult to see that
$S(\mathcal{M})=S(\mathcal{M},\tau)=\mathcal{M}$ (see \cite{LSZ,F} for more details).
In this case, for $A\in S(\mathcal{M},\tau),$ we have
$$\mu(n,A)=\mu(t,A), \,\,\,  t\in[n,n+1), \, n\in\mathbb{Z}_{+}.$$
The sequence $\{\mu(n,A)\}_{n\in\mathbb{Z}_{+}}$ is just the sequence of singular values of the operator $A\in B(H).$

\subsection{Symmetric (Quasi-)Banach Function and Operator Spaces}

\begin{definition}\label{Sym} We say that $(E,\|\cdot\|_E)$ is a symmetric (quasi-)Banach function (or sequence) space on $I$ if the following hold:
\begin{enumerate}[{\rm (a)}]
\item $E$ is a subset of $S(I,m);$
\item $(E,\|\cdot\|_E)$ is a (quasi-)Banach space;
\item If $x\in E$ and if $y\in S(I,m)$  are such that $|y|\leq|x|,$ then $y\in E$ and $\|y\|_E\leq\|x\|_E;$
\item If $x\in E$ and if $y\in S(I,m)$ are such that $\mu(y)=\mu(x),$ then $y\in E$ and $\|y\|_E=\|x\|_E.$
\end{enumerate}
\end{definition}
For the general theory of symmetric spaces, we refer the reader to \cite{BSh,KPS,LT}.

Let $\mathcal{M}$ be a semifinite von Neumann algebra on a Hilbert space $H$ equipped with a faithful normal semifinite trace $\tau.$ The von Neumann algebra
$\mathcal{M}$ is also denoted by $\mathcal{L}_{\infty}(\mathcal{M}),$ and for all $A\in \mathcal{L}_{\infty}(\mathcal{M}),$ $\|A\|_{\mathcal{L}_{\infty}(\mathcal{M})}:=\|A\|_{B(H)}$.
Moreover, $\|A\|_{\mathcal{L}_{\infty}(\mathcal{M})}=\|\mu(A)\|_{L_{\infty}(0,\infty)}=\mu(0,A)$, $A\in \mathcal{L}_{\infty}(\mathcal{M})$.

\begin{definition}\label{NC Sym} Let $\mathcal{E}(\mathcal{M},\tau)$ be a linear subset in $S(\mathcal{M},\tau)$ equipped with a complete (quasi-)norm $\|\cdot\|_{\mathcal{E}(\mathcal{M},\tau)}.$ 
We say that $\mathcal{E}(\mathcal{M},\tau)$ is a symmetric operator space (on $\mathcal{M}$, or in $S(\mathcal{M},\tau)$) if
for $A\in \mathcal{E}(\mathcal{M},\tau)$ and for every $B\in S(\mathcal{M},\tau)$ with $\mu(B)\leq\mu(A),$ we have $B\in \mathcal{E}(\mathcal{M},\tau)$ and $\|B\|_{\mathcal{E}(\mathcal{M},\tau)}\leq\|A\|_{\mathcal{E}(\mathcal{M},\tau)}$.

A symmetric function (or sequence) space is the term reserved for a symmetric
operator space when $\mathcal{M}=L_{\infty}(I,m),$ where $I=(0,\infty), \mathbb{R}$ (or $\mathcal{M}=\ell_{\infty}(I)$ with counting measure, where $I=\mathbb{Z}_{+},\mathbb{Z}$).
\end{definition}

Recall the construction of a symmetric (quasi-)Banach operator space (or non-commutative symmetric (quasi-)Banach space) $\mathcal{E}(\mathcal{M},\tau).$ Let $E$ be a symmetric (quasi-)Banach function (or sequence) space on $(0,\infty)$ (or $\mathbb{Z}_{+}$). Set
$$\mathcal{E}(\mathcal{M},\tau)=\Big\{A\in S(\mathcal{M},\tau):\ \mu(A)\in E\Big\}.$$
We equip $\mathcal{E}(\mathcal{M},\tau)$ with a natural norm
$$\|A\|_{\mathcal{E}(\mathcal{M},\tau)}:=\|\mu(A)\|_E,\quad A\in \mathcal{E}(\mathcal{M},\tau).$$
This is a (quasi-)Banach space with the (quasi-)norm $\|\cdot\|_{\mathcal{E}(\mathcal{M},\tau)}$ and is called the (non-commutative) symmetric operator space associated with $(\mathcal{M},\tau)$ corresponding to $(E,\|\cdot\|_{E})$.
In particular, when $\mathcal{M}=B(H),$ we have
$$\mathcal{E}(H,Tr)=\Big\{A\in B(H):\ \mu(A)\in E(\mathbb{Z}_{+})\Big\}.$$
An extensive discussion of the various properties of such spaces can be found
in \cite{KS,LSZ}.
Futhermore, the following fundamental theorem was proved in \cite{KS} (see also \cite[Question 2.5.5, p. 58]{LSZ}).
\begin{thm} Let $E$ be a symmetric function (or sequence) space on $(0,\infty)$ (or $\mathbb{Z}_{+}$) and let $\mathcal{M}$ be a semifinite von Neumann algebra. Set
$$\mathcal{E}(\mathcal{M},\tau)=\Big\{A\in S(\mathcal{M},\tau):\ \mu(A)\in E\Big\}.$$
So defined $(\mathcal{E}(\mathcal{M},\tau),\|\cdot\|_{\mathcal{E}(\mathcal{M},\tau)})$
is a symmetric operator space.
\end{thm}

 The (quasi-)norm on $E$ is order continuous in the sense that $0\leq x_\beta\downarrow_{\beta}0$ in $E$ implies that $\|x_{\beta}\|_{E}\downarrow 0$, and in this case, the (quasi-)norm is order continuous on the  non-commutative space $\mathcal{E}(\mathcal{M},\tau)$ (see \cite[Chapter IV]{DdePS} for more details).
For simplicity, throughout this paper we denote $\mathcal{E}(\mathcal{M},\tau)$ (resp. $\mathcal{E}(H,Tr)$) by $\mathcal{E}(\mathcal{M})$ (resp. $\mathcal{E}(H)$).

For $1\leq p<\infty,$ we set
$$\mathcal{L}_p(\mathcal{M})=\{A\in S(\mathcal{M},\tau):\ \tau(|A|^p)<\infty\},\quad \|A\|_{\mathcal{L}_p(\mathcal{M})}=(\tau(|A|^p))^{\frac1p}.$$
The Banach spaces $(\mathcal{L}_p(\mathcal{M}),\|\cdot\|_{\mathcal{L}_p(\mathcal{M})})$ ($1\leq p<\infty$) are separable.
Moreover, when $p=2$ this space
becomes Hilbert space with the inner product
$$<A,B>:=\tau(B^{*}A),\,\, A,B\in \mathcal{L}_{2}(\mathcal{M}),$$  where $B^{*}$ is adjoint operator of $B.$
It is easy to see that these spaces are symmetric spaces. In particular, when $\mathcal{M}=B(H),$ we denote $\mathcal{L}_p(\mathcal{M})$ by $\mathcal{L}_p(H).$

The dilation operator on $S(0,\infty)$ is defined by
$$\sigma_{s}x(t):=x\left(\frac{t}{s}\right), \,\,\,\ s>0.$$
It is obvious that the dilation operator $\sigma_{s}$ is continuous in $S(0,\infty)$ (see \cite[Chapter II.3, p. 96]{KPS}).
The discrete dilation operator $\sigma_m,$ $m\in\mathbb{N},$ on $\ell_{\infty}(\mathbb{Z}_{+})$ is defined by
$$\sigma_{m}(a):=\{\underbrace{a(0),a(0),...,a(0)}_{m \,\ \text{terms}},\underbrace{a(1),a(1),...,a(1)}_{m \,\ \text{terms}},...,\underbrace{a(n),a(n),...,a(n)}_{m \,\ \text{terms}},...\}.$$

\subsection{K\"{o}the dual of Symmetric Function and Operator spaces}\label{Kote dual}

Next we define the K\"{o}the dual space of symmetric Banach function spaces.
Given a symmetric Banach function (or sequence) space $E$  on $(0,\infty)$ (or $\mathbb{Z}_{+}$), equipped
with Lebesgue measure $m$ (or counting measure) the K\"{o}the dual space $E^\times$ on $(0,\infty)$ (or $\mathbb{Z}_{+}$) is defined by
$$E(0,\infty)^\times=\left\{y\in S(0,\infty):\int_{0}^{\infty}|x(t)y(t)|dt<\infty, \,\ \forall x\in E(0,\infty)\right\}$$
$$\left( \text{or} \,\,\ E(\mathbb{Z}_{+})^\times=\left\{y\in \ell_{\infty}(\mathbb{Z}_{+}):\sum_{k=0}^{\infty}|x(k)y(k)|<\infty, \,\ \forall x\in E(\mathbb{Z}_{+})\right\}\right).$$
The space $E^\times$ is Banach with the norm
\begin{equation}\label{kote dual norm}\|y\|_{E(0,\infty)^\times}:=\sup\left\{\int_{0}^{\infty}|x(t)y(t)|dt:x\in E(0,\infty), \,\ \|x\|_{E(0,\infty)}\leq1\right\}
\end{equation}
$$\left( \text{or} \,\, \|y\|_{E(\mathbb{Z}_{+})^\times}:=\sup\left\{\sum_{k=0}^{\infty}|x(k)y(k)|:x\in E(\mathbb{Z}_{+}), \,\ \|x\|_{E(\mathbb{Z}_{+})}\leq1\right\}\right).
$$

Let $E$ and $F$ be symmetric spaces on $(0,\infty)$ (or $\mathbb{Z}_{+}$). The following important properties of K\"{o}the duals of symmetric spaces $E$ and $F$ follow from Lozanovski\v{i}'s work \cite{Loz}
\begin{equation}\label{sum and inter dual}
(E+F)^{\times}=E^{\times}\cap F^{\times}, \, (E\cap F)^{\times}=E^{\times}+F^{\times}.
\end{equation}

If $E$ is a symmetric Banach function (or sequence) space, then $(E^{\times},\|\cdot\|_{E^{\times}})$ is also a symmetric Banach function (or sequence) space (cf. \cite[Section 2.4]{BSh}). For more details on K\"{o}the duality we refer to \cite{BSh,LT}.

Next we give the definition of  K\"{o}the dual space of non-commutative symmetric spaces. We assume that $E$ is a symmetric Banach function (resp. sequence) space
on $(0,\infty)$ (resp. $\mathbb{Z}_{+}$).
\begin{definition} \label{duality} The K\"{o}the dual space $\mathcal{E}(\mathcal{M},\tau)^\times$ of $\mathcal{E}(\mathcal{M},\tau)$ is defined
$$\mathcal{E}(\mathcal{M},\tau)^\times=\{A\in S(\mathcal{M},\tau):AB\in L_{1}(\mathcal{M},\tau), \,\, \forall B\in \mathcal{E}(\mathcal{M},\tau)\}.$$
\end{definition}
It is clear that $\mathcal{E}(\mathcal{M},\tau)^\times$ is a linear subspace of $S(\mathcal{M},\tau).$
Moreover, it is a normed space with the norm defined by setting
$$\|A\|_{\mathcal{E}(\mathcal{M},\tau)^\times}:=\sup\{|\tau(AB)|: \, B\in \mathcal{E}(\mathcal{M},\tau), \, \|B\|_{\mathcal{E}(\mathcal{M},\tau)}\leq 1\}.$$
It can be shown that the normed space $(\mathcal{E}(\mathcal{M},\tau)^\times,\|\cdot\|_{\mathcal{E}(\mathcal{M},\tau)^\times})$ is complete with respect to $\|\cdot\|_{\mathcal{E}(\mathcal{M},\tau)^\times}$
(see for instance \cite[Proposition 5.2 (x) and 5.4 ]{DDdeP1}).
The important result for the identification of the K\"{o}the dual $ \mathcal{E}(\mathcal{M},\tau)^\times$ is that if $E$ is a symmetric Banach function space on $(0,\infty)$ (or $\mathbb{Z}_{+}$) with K\"{o}the dual space $ E^{\times},$
then $\mathcal{E}(\mathcal{M},\tau)^\times=\mathcal{E}^{\times}(\mathcal{M},\tau)$ (with equality of norms) (see for instance \cite[Theorem 5.6]{DDdeP1}). For more details on non-commutative K\"{o}the dual spaces we refer the reader to
\cite{DdeP,DDdeP1},\cite[Chapter IV]{DdePS}.
\begin{proposition}\label{Holder} Let $\mathcal{E}(\mathcal{M})$ be a symmetric Banach operator space with K\"{o}the dual $\mathcal{E}(\mathcal{M})^{\times}.$ If $X\in\mathcal{ E}(\mathcal{M})$ and $Y\in \mathcal{E}(\mathcal{M})^{\times},$ then $XY\in \mathcal{L}_{1}(\mathcal{M})$ and
$$\tau(|XY|)\leq \|X\|_{\mathcal{E}(\mathcal{M})}\cdot\|Y\|_{\mathcal{E}(\mathcal{M})^{\times}}.$$
\end{proposition}

\subsection{$\mathcal{L}_{1}\cap \mathcal{L}_{\infty}$ and $\mathcal{L}_{1}+\mathcal{L}_{\infty}$ spaces}

Two examples below are of particular interest. Consider the separated topological vector space $S(0,\infty)$ consisting of all measurable functions $x$ such that $m(\{t : |x(t)| > s\})$ is finite for some $s > 0$ with the topology of convergence in measure. Then the spaces $L_{1}(0,\infty)$ and $L_{\infty}(0,\infty)$ are algebraically and topologically imbedded in the topological vector space $S(0,\infty)$, and so these spaces form a Banach couple (see \cite[Chapter I]{KPS} for more details). The space $(L_{1}\cap L_{\infty})(0,\infty)=L_{1}(0,\infty)\cap L_{\infty}(0,\infty)$ consists of all bounded summable functions $x$ on $(0,\infty)$ with the norm
$$\|x\|_{(L_{1}\cap L_{\infty})(0,\infty)}=\max\{\|x\|_{L_{1}(0,\infty)},\|x\|_{L_{\infty}(0,\infty)}\}, x\in(L_{1}\cap L_{\infty})(0,\infty).$$
The space $(L_{1}+L_{\infty})(0,\infty)=L_{1}(0,\infty)+ L_{\infty}(0,\infty)$ consists of functions which are sums of bounded measurable and summable functions $x\in S(0,\infty)$
equipped with the norm given by
$$\|x\|_{(L_{1}+L_{\infty})(0,\infty)}=\inf\{\|x_1\|_{L_{1}(0,\infty)}+\|x_2\|_{L_{\infty}(0,\infty)}:x=x_1+x_2, $$
$$\, x_{1}\in L_{1}(0,\infty), \, x_2\in L_{\infty}(0,\infty)\}.$$
For more details we refer the reader to \cite[Chapter I]{BSh},\cite[Chapter II]{KPS}.
We recall that that every symmetric Banach function space on $(0,\infty)$ (with respect to Lebesgue measure) satisfies
$$(L_1\cap L_{\infty})(0,\infty)\subset E(0,\infty)\subset (L_1+L_{\infty})(0,\infty)$$
equipped with the norm given by
with continuous embeddings (see for instance \cite[Theorem II. 4.1. p. 91]{KPS}).

We define the space $\mathcal{L}_{1}(\mathcal{M})+\mathcal{L}_{\infty}(\mathcal{M})$ as the class of those operators $A\in S(\mathcal{M},\tau)$ for which
\begin{eqnarray*}\begin{split}\|A\|_{\mathcal{L}_{1}(\mathcal{M})+\mathcal{L}_{\infty}(\mathcal{M})}
&:=\inf\{\|A_1\|_{\mathcal{L}_1(\mathcal{M})}+\|A_2\|_{\mathcal{L}_{\infty}(\mathcal{M})}:\\
&A=A_1+A_2,\, A_1\in \mathcal{L}_{1}(\mathcal{M}), A_2\in \mathcal{L}_{\infty}(\mathcal{M})\}<\infty.
\end{split}\end{eqnarray*}
\begin{thm}\cite[Theorem III. 9.16, p. 96]{DdePS}\label{s-calcul thm} If $A\in S(\mathcal{M},\tau),$ then
\begin{equation}\label{L+L}
\|A\|_{\mathcal{L}_{1}(\mathcal{M})+\mathcal{L}_{\infty}(\mathcal{M})}=\int_{0}^{1}\mu(t,A)dt, \,\,\,\,\, A\in \mathcal{L}_{1}(\mathcal{M})+\mathcal{L}_{\infty}(\mathcal{M}).
\end{equation}
\end{thm}

A standard argument shows that $\mathcal{L}_{1}(\mathcal{M})+\mathcal{L}_{\infty}(\mathcal{M})$ is a Banach space with respect to this norm, which is continuously embedded in $S(\mathcal{M},\tau).$ It follows from Theorem \ref{s-calcul thm} that if $A\in S(\mathcal{M},\tau),$ then $A\in
\mathcal{L}_{1}(\mathcal{M})+\mathcal{L}_{\infty}(\mathcal{M})$ if and only if $\int_{0}^{t}\mu(s,A)ds<\infty$ for all $t>0.$

In particular, if $A\in S(\mathcal{M},\tau),$ then $A\in \mathcal{L}_{1}(\mathcal{M})+\mathcal{L}_{\infty}(\mathcal{M})$ if and only if $\mu(A)\in (L_{1}+L_{\infty})(0,\infty)$ and
$$\|A\|_{\mathcal{L}_{1}(\mathcal{M})+\mathcal{L}_{\infty}(\mathcal{M})}=\|\mu(A)\|_{(L_{1}+L_{\infty})(0,\infty)}, \,\,\,\,\, A\in \mathcal{L}_{1}(\mathcal{M})+\mathcal{L}_{\infty}(\mathcal{M}).$$
The space $\mathcal{L}_{1}(\mathcal{M})+\mathcal{L}_{\infty}(\mathcal{M})$ is also denoted by $(\mathcal{L}_{1}+\mathcal{L}_{\infty})(\mathcal{M}).$
Similarly, define the intersection space of Banach spaces $\mathcal{L}_{1}(\mathcal{M})$ and $\mathcal{L}_{\infty}(\mathcal{M})$ as follows
$$(\mathcal{L}_{1}\cap \mathcal{L}_{\infty})(\mathcal{M})=\{A\in S(\mathcal{M},\tau): \,\ \|A\|_{(\mathcal{L}_{1}\cap \mathcal{L}_{\infty})(\mathcal{M})}<\infty\}, $$
where the norm on $(\mathcal{L}_{1}\cap \mathcal{L}_{\infty})(\mathcal{M})$ is defined by
$$
\|A\|_{(\mathcal{L}_{1}\cap \mathcal{L}_{\infty})(\mathcal{M})}=\max\{\|A\|_{\mathcal{L}_{1}(\mathcal{M})},\|A\|_{\mathcal{L}_{\infty}(\mathcal{M})}\}, \,\,\  A\in (\mathcal{L}_{1}\cap \mathcal{L}_{\infty})(\mathcal{M}).
$$
It is easy to see that $(\mathcal{L}_{1}\cap \mathcal{L}_{\infty})(\mathcal{M})$ is a Banach spaces with respect to this norm. It should be observed that if $A\in S(\mathcal{M},\tau),$ then $A\in (\mathcal{L}_{1}\cap \mathcal{L}_{\infty})(\mathcal{M})$ if and only
if $\mu(A)\in (L_{1}\cap L_{\infty})(0,\infty).$ Moreover,
$$\|A\|_{(\mathcal{L}_{1}\cap \mathcal{L}_{\infty})(\mathcal{M})}=\|\mu(A)\|_{(L_{1}\cap L_{\infty})(0,\infty)}, \,\,\,\, A\in (\mathcal{L}_{1}\cap \mathcal{L}_{\infty})(\mathcal{M}). $$
For an in-depth description on these spaces, refer to \cite[Chapter III]{DdePS} (see also \cite{DDdeP1}).
As in the commutative case, every symmetric Banach operator space satisfies
\begin{equation}\label{imbed}
(\mathcal{L}_1\cap \mathcal{L}_{\infty})(\mathcal{M})\subset \mathcal{E}(\mathcal{M})\subset (\mathcal{L}_1+\mathcal{L}_{\infty})(\mathcal{M})
\end{equation}
(see \cite[Theorem 4.5, pp. 36-37]{DdePS} for more details).

\subsection{Lorentz spaces}\label{NC Lorentz}

\begin{definition}\cite[Definition II. 1.1, p.49]{KPS} A function $\varphi$ on the semiaxis $[0,\infty)$ is said to be quasiconcave if
\begin{enumerate}[{\rm (i)}]
         \item $\varphi(t)=0\Leftrightarrow t=0;$
          \item $\varphi(t)$ is positive and increasing for $t>0;$
         \item $\frac{\varphi(t)}{t}$ is decreasing for $t>0.$
\end{enumerate}
\end{definition}
Observe that every nonnegative concave function on $[0,\infty)$ that vanishes only at origin is quasiconcave. The reverse, however, is not always true.
But, we may replace, if necessary, a quasiconcave function $\varphi$ by its least concave majorant $\widetilde{\varphi}$ such that
$$\frac{1}{2}\widetilde{\varphi}\leq\varphi\leq\widetilde{\varphi}$$
(see \cite[Proposition 5.10, p. 71]{BSh}).

Let $\Omega$ denote the set of increasing concave functions $\varphi:[0,\infty)\rightarrow[0,\infty)$ for which $\lim_{t\rightarrow 0+}\varphi(t)=0$ (or simply $\varphi(+0)=0$).
For the function $\varphi$ in $\Omega,$ the Lorentz space $\Lambda_{\varphi}(0,\infty)$ is defined by setting
$$\Lambda_{\varphi}(0,\infty):=\left\{x\in S(0,\infty): \int_0^\infty\mu(s,x)d\varphi(s)<\infty\right\}$$
and equipped with the norm
$$\|x\|_{\Lambda_\varphi(0,\infty)}:=\int_0^\infty\mu(s,x)d\varphi(s).$$
The Lorentz sequence space $\Lambda_{\varphi}(\mathbb{Z}_{+})$ is
$$\Lambda_{\varphi}(\mathbb{Z}_{+}):=\left\{a\in c_{0}:\|a\|_{\Lambda_{\varphi}(\mathbb{Z}_{+})}=\sum_{n=0}^{\infty}\mu(n,a)(\varphi(n+1)-\varphi(n))<\infty\right\},$$
where $c_{0}$ is the space of sequences converging to zero. These spaces are examples of symmetric Banach function spaces.
In particular, if $\varphi(t):=\log(1+t) \, (t>0),$ then the Lorentz sequence space $\Lambda_{\log}(\mathbb{Z}_{+})$ is defined as follows:
$$\Lambda_{\log}(\mathbb{Z}_{+}):=\left\{a\in c_{0}:\|a\|_{\Lambda_{\log}(\mathbb{Z}_{+})}=\sum_{n=0}^{\infty}\frac{\mu(n,a)}{n+1}<\infty\right\}.$$
These spaces are defined similarly on $\mathbb{R}$ and $\mathbb{Z},$ respectively.
For more details on Lorentz spaces, we refer the reader to \cite[Chapter II.5]{BSh} and \cite[Chapter II.5]{KPS}.

As in the commutative case, for a function $\varphi$ in $\Omega$ define the corresponding non-commutative Lorentz space by setting
$$\Lambda_{\varphi}(\mathcal{M}):=\left\{A\in S(\mathcal{M},\tau): \int_0^\infty\mu(s,A)d\varphi(s)<\infty\right\}$$
equipped with the norm
\begin{equation}\label{Lphi}
\|A\|_{\Lambda_\varphi(\mathcal{M})}:=\int_0^\infty\mu(s,A)d\varphi(s).
\end{equation}
These operator spaces become symmetric operator spaces.
The Lorentz ideal $\Lambda_{\varphi}(H)$ (see \cite[Example 1.2.7, p. 25]{LSZ}) is
$$\Lambda_{\varphi}(H):=\left\{A\in \mathcal{K}(H):\|A\|_{\Lambda_{\varphi}(H)}=\sum_{n=0}^{\infty}\mu(n,A)(\varphi(n+1)-\varphi(n))<\infty\right\},$$
where $\mathcal{K}(H)$ is the ideal of compact operators on $H.$
If $\varphi(t):=\log(1+t)\, (t>0),$ then the corresponding Lorentz ideal $\Lambda_{\log}(H)$ is defined by
$$\Lambda_{\log}(H):=\left\{A\in \mathcal{K}(H):\|A\|_{\Lambda_{\log}(H)}=\sum_{n=0}^{\infty}\frac{\mu(n,A)}{n+1}<\infty\right\}.$$
This ideal contains all Schatten-von Neumann classes $\mathcal{L}_{p}(H)$ $(1\leq p<\infty).$ It corresponds to the double index $(\infty,1)$ on the Lorentz scale and is known as the Macaev ideal (see \cite{GK1}).

\subsection{Weak $L_{1}$ and $M_{1,\infty}$ spaces}\label{weak L1}

The weak-$\ell_1$ sequence space $\ell_{1,\infty}$ on $\mathbb{Z}_{+}$ (or $\mathbb{Z}$) is defined as
$$\ell_{1,\infty}:=\left\{a\in c_{0}:\mu(n,a)=O\left(\frac{1}{1+n}\right)\right\}.$$

Further, define the space $\mathcal{L}_{1,\infty}(\mathcal{M})$ by setting
$$\mathcal{L}_{1,\infty}(\mathcal{M})=\{A\in S(\mathcal{M},\tau):\ \sup_{t>0}t\mu(t,A)<\infty\},$$
and equip $\mathcal{L}_{1,\infty}(\mathcal{M})$ with the functional $\|\cdot\|_{\mathcal{L}_{1,\infty}(\mathcal{M})}$ defined by the formulae

\begin{equation}\label{weakL1}
\|A\|_{\mathcal{L}_{1,\infty}(\mathcal{M})}=\sup_{t>0}t\mu(t,A),\quad A\in \mathcal{L}_{1,\infty}(\mathcal{M}).
\end{equation}
For any $A,B\in \mathcal{L}_{1,\infty}(\mathcal{M}),$ it follows from \eqref{triangle svf} that
\begin{eqnarray*}
\begin{split}\|A+B\|_{\mathcal{L}_{1,\infty}(\mathcal{M})}=\sup_{t>0}t\mu(t,A+B)
\leq\sup_{t>0}t\left(\mu\left(\frac{t}{2},A\right)+\mu\left(\frac{t}{2},B\right)\right) \\
\leq\sup_{t>0}t\mu\left(\frac{t}{2},A\right)+\sup_{t>0}t\mu\left(\frac{t}{2},B\right)
=2\|A\|_{\mathcal{L}_{1,\infty}(\mathcal{M})}+2\|B\|_{\mathcal{L}_{1,\infty}(\mathcal{M})}.
\end{split}
\end{eqnarray*}
 That is, $\|\cdot\|_{\mathcal{L}_{1,\infty}(\mathcal{M})}$ is a quasi-norm. In fact, the space $(\mathcal{L}_{1,\infty}(\mathcal{M}),\|\cdot\|_{\mathcal{L}_{1,\infty}(\mathcal{M})})$ is quasi-Banach (see \cite[Section 7]{KS} or \cite{F}).
The ideal $\mathcal{L}_{1,\infty}(\mathcal{M})$ has the Fatou property. That is, if $A_n\in \mathcal{L}_{1,\infty}(\mathcal{M}),$ $\|A_n\|_{\mathcal{L}_{1,\infty}(\mathcal{M})}\leq 1$ and $A_n\rightarrow A$ in measure, then $A\in \mathcal{L}_{1,\infty}(\mathcal{M})$  and $\|A\|_{\mathcal{L}_{1,\infty}(\mathcal{M})}\leq 1$.

Let $E$ be a quasi-Banach symmetric sequence space on $\mathbb{Z}_{+}.$
Let $\ell_{1,\infty}(\mathbb{Z}_{+})$ be a weak-$\ell_1$ sequence space. Define corresponding weak-$\mathcal{L}_1$ ideal of compact operators on $H$
$$\mathcal{L}_{1,\infty}(H):=\left\{A\in \mathcal{K}(H):\mu(n,A)=O(\frac{1}{1+n})\right\},$$
(see \cite[Example 1.2.6, p. 24]{LSZ}).

Define the Marcinkiewicz (or Lorentz) space $M_{1,\infty}(0,\infty)$ by setting
$$M_{1,\infty}(0,\infty):=\left\{x\in S(0,\infty): \sup_{t>0}\frac{1}{\log(1+t)}\int_0^t\mu(s,x)ds<\infty\right\}$$
equipped with the norm
$$\|x\|_{M_{1,\infty}(0,\infty)}:=\sup_{t>0}\frac{1}{\log(1+t)}\int_0^t\mu(s,x)ds.$$
This is an example of a symmetric Banach function space. For more information on Marcinkiewicz spaces we refer the reader to \cite[Chapter II.5]{BSh} and \cite[Chapter II.5]{KPS}.
Similarly, the non-commutative $\mathcal{M}_{1,\infty}(\mathcal{M})$ space is defined by
$$\mathcal{M}_{1,\infty}(\mathcal{M}):=\left\{A\in S(\mathcal{M},\tau): \sup_{t>0}\frac{1}{\log(1+t)}\int_0^t\mu(s,A)ds<\infty\right\}$$
equipped with the norm
$$\|A\|_{\mathcal{M}_{1,\infty}(\mathcal{M})}:=\sup_{t>0}\frac{1}{\log(1+t)}\int_0^t\mu(s,A)ds.$$
These spaces are also symmetric spaces. We refer to \cite{LSZ} (see also \cite{CRSS}) for detailed study of this space and its applications in non-commutative geometry.

Let $\mathcal{M}=B(H)$ and $\tau=Tr.$ Then, we have
$$\mathcal{M}_{1,\infty}(H):=\left\{A\in \mathcal{K}(H):\|A\|_{\mathcal{M}_{1,\infty}(H)}:=\sup_{n\in\mathbb{Z}_{+}}\frac{1}{\log(2+n)}\sum_{k=0}^{n}\mu(k,A)<\infty \right\},$$
where $\{\mu(k,A)\}_{k\in \mathbb{Z}_{+}}$ is the sequence of singular values of a compact operator $A.$ This space is known as the dual of the Macaev ideal on a separable Hilbert space $H.$

Moreover, the space $\mathcal{M}_{1,\infty}(H)$ contains the quasi-Banach ideal $\mathcal{L}_{1,\infty}(H)$ of compact operators, i.e. the following inclusion
$$\mathcal{L}_{1,\infty}(H)\subset \mathcal{M}_{1,\infty}(H)$$
is strict (see \cite[Lemma 1.2.8 and Example 1.2.9, pp. 25-26]{LSZ}).

For $\mathcal{M}=\ell_{\infty}(\mathbb{Z}_{+}),$ the space $M_{1,\infty}(\mathbb{Z}_{+})$ is defined by
\begin{equation}\label{Marsecuence}
M_{1,\infty}(\mathbb{Z}_{+}):=\{a\in c_{0}:\|a\|_{M_{1,\infty}(\mathbb{Z}_{+})}:=\sup_{n\geq 0}\frac{1}{\log(2+n)}\sum_{k=0}^{n}\mu(k,a)<\infty\},
\end{equation}
(see \cite[Example 1.2.7, p. 24]{LSZ}).

\subsection{Calder\'{o}n operator and Hilbert transform}\label{calderon}
Recall that $S(0,\infty)$ is the space of all Lebesgue measurable functions on $(0,\infty)$ such that $m(\{t : |x(t)| > s\})$ is finite for some $s > 0.$
Define operators $C:(L_{1}+L_{\infty})(0,\infty)\rightarrow (L_{1,\infty}+L_{\infty})(0,\infty)$ by
\begin{equation}\label{Ces}(Cx)(t):=\frac{1}{t}\int_{0}^{t}x(s)ds , \,\ x\in(L_{1}+L_{\infty})(0,\infty)
\end{equation}
and $C':\Lambda_{\log}(0,\infty)\rightarrow S(0,\infty) $ by
$$(C'x)(t):=\int_t^{\infty}x(s)\frac{ds}{s}, \,\, x\in\Lambda_{\log}(0,\infty),$$
where $C$ is called the Ces\`{a}ro operator (or else Hardy-Littlewood operator or Hardy operator as in \cite[Chapter II.3]{BSh},\cite[Chapter II.6]{KPS}).

For each $x\in \Lambda_{\log}(0,\infty),$ define the Calder\'{o}n operator $S:\Lambda_{\log}(0,\infty)\rightarrow(L_{1,\infty}+L_{\infty})(0,\infty)$ as a sum of $C$ and $C'$ by
\begin{equation}\label{S}
 (Sx)(t):=\frac1t\int_0^tx(s)ds+\int_t^{\infty}x(s)\frac{ds}{s}=(Cx)(t)+(C'x)(t), \,\ x\in\Lambda_{\log}(0,\infty).
\end{equation}
It is obvious that $S$ is linear operator.
If $0<t_{1}<t_{2},$ then
$$\min\left(1,\frac{s}{t_{2}}\right)\leq \min\left(1,\frac{s}{t_{1}}\right)\leq \frac{t_{2}}{t_{1}}\cdot\min\left(1,\frac{s}{t_{2}}\right), \,\, s>0.$$
Therefore, if $x$ is nonnegative, it follows from the first of these inequalities that $(Sx)(t)$ is a
decreasing function of $t.$ The operator $S$ is often applied to the decreasing rearrangement $\mu(x)$ of a function $x$ defined on some other measure space.
Since $S\mu(x)$ is itself decreasing, it is easy to see that $\mu(S\mu(x))=S\mu(x).$
Let $x\in \Lambda_{\log}(0,\infty).$ Since for each $t>0,$ the kernel
$k_{t}(s)=\frac{1}{s}\cdot\min\Big\{1,\frac{s}{t}\Big\}$
is a decreasing function of $s,$ it follows  from \cite[Theorem II.2.2, p. 44]{BSh} that
\begin{equation}\label{S proper}\begin{split}
|(Sx)(t)|
&\stackrel{\eqref{S}}{=}\Big|\int_{0}^{\infty}x(s)\min\Big\{1,\frac{s}{t}\Big\}\frac{ds}{s}\Big|\\
&\leq\int_{0}^{\infty}|x(s)|\min\Big\{1,\frac{s}{t}\Big\}\frac{ds}{s}\leq\int_{0}^{\infty}\mu(s,x)\min\Big\{1,\frac{s}{t}\Big\}\frac{ds}{s}\stackrel{\eqref{S}}{=}(S\mu(x))(t).
\end{split}\end{equation}
For more information on these operators, we refer to \cite[Chapter III]{BSh} and \cite[Chapter II]{KPS}.

If $x\in \Lambda_{\log}(\mathbb{R}),$  then the classical Hilbert transform $\mathcal{H}$ is defined by the principal-value integral
\begin{equation}\label{hilbert tr}
(\mathcal{H}x)(t)=p.v.\frac{1}{\pi}\int_{\mathbb{R}}\frac{x(s)}{t-s}ds, \,\,\ x\in\Lambda_{\log}(\mathbb{R}).
\end{equation}

\begin{rmk}\label{dom of Hilbert trans} Let $x=x\chi_{(0,\infty)}$ such that $x$ is a non-negative decreasing function on $(0,\infty).$ Then it is easy to see that
\begin{eqnarray*}\begin{split} |(\mathcal{H}x)(-t)|
&\stackrel{\eqref{hilbert tr}}{=}\frac{1}{\pi}\left|\int_{\mathbb{R}_{+}}\frac{x(s)}{-t-s}ds\right|\\
&=\frac{1}{\pi}\int_{\mathbb{R}_{+}}\frac{x(s)}{t+s}ds=\frac{1}{\pi}\left(\int_{0}^{t}\frac{x(s)}{t+s}ds+\int_{t}^{\infty}\frac{x(s)}{t+s}ds\right)\\
&\geq\frac{1}{\pi}\left(\int_{0}^{t}\frac{x(s)}{2t}ds+\int_{t}^{\infty}\frac{x(s)}{2s}ds\right)\\
&=\frac{1}{2\pi}\cdot\left(\frac{1}{t}\int_{0}^{t}x(s)ds+\int_{t}^{\infty}\frac{x(s)}{s}ds\right)\stackrel{\eqref{S}}{=}\frac{1}{2\pi}(Sx)(t), \,\, t>0,
\end{split}\end{eqnarray*}
i.e. we have
$$\frac{1}{2\pi}(S\mu(x))(t)\leq |(\mathcal{H}x)(-t)|, \,\, t>0.$$
Therefore, if $(\mathcal{H}x)(-t)$ exists, then it follows that $S\mu(x)$ exists, and it means $x$ belongs to the domain of $S,$ i.e. $x\in \Lambda_{\log}(0,\infty)$ (see \eqref{S}).
On the other hand, if $x\in\Lambda_{\log}(0,\infty),$ then by \cite[Theorem III.4.8, p. 138]{BSh}, we have
$$\mu(\mathcal{H}x)\leq c_{abs}S\mu(x),$$
which shows existence of $\mathcal{H}x.$
\end{rmk}

Let $\mathcal{M}\bar{\otimes}L_{\infty}(\mathbb{R})$ be a von Neumann tensor product equipped with the normal semifinite faithful tensor product trace $\nu=\tau\otimes m,$ where $m$ is the
trace on $L_{\infty}(\mathbb{R})$ given by integration with respect to Lebesgue measure on $\mathbb{R}.$ Let $E$ be a symmetric quasi-Banach space on $\mathbb{R}.$

If the operator $1\otimes \mathcal{H}$ is defined on some $\mathcal{E}(\mathcal{M}\bar{\otimes}L_{\infty}(\mathbb{R})),$ then $\mathcal{H}$ is defined on $E(\mathbb{R}).$ By Remark \ref{dom of Hilbert trans}, it must be that $E(\mathbb{R})\subset\Lambda_{\log}(\mathbb{R}).$ So, $1\otimes \mathcal{H}$ cannot be defined outside of $\Lambda_{\log}(\mathcal{M}\bar{\otimes}L_{\infty}(\mathbb{R})).$
Set
$$1\otimes \mathcal{H}:\sum_{k=1}^{n} x_k\otimes f_k\to\sum_{k=1}^{n} x_k\otimes \mathcal{H}(f_k),$$
where $x_{k}\in \mathcal{M}$ and $f_{k}\in \Lambda_{\log}(\mathbb{R}),$ $k=1,2,...,n.$
Then, these elementary tensors are dense in $\Lambda_{\log}(\mathcal{M}\bar{\otimes}L_{\infty}(\mathbb{R}))$ and we have norm estimate.
Therefore, $1\otimes \mathcal{H}$ is defined on $\Lambda_{\log}(\mathcal{M}\bar{\otimes}L_{\infty}(\mathbb{R})).$

Define the discrete version of the operator $S^{d}:\Lambda_{\log}(\mathbb{Z}_{+})\rightarrow(\ell_{1,\infty}+\ell_{\infty})(\mathbb{Z}_{+})$ by
\begin{equation}\label{S dis}\big(S^{d}a\big)(n):=\frac{1}{n+1}\sum_{k=0}^{n}a(k)+\sum_{k=n+1}^{+\infty}\frac{a(k)}{k}, \,\,\ a\in \Lambda_{\log}(\mathbb{Z}_{+}).
\end{equation}

\subsection{Triangular truncation operator}\label{triangular}

Our primary example is a triangular truncation operator on the Hilbert space $H=L_2(\mathbb{R}).$ More precisely, let $K$ be a fixed measurable function on $\mathbb{R}\times\mathbb{R}.$ Let us consider an operator $V$ with the
integral kernel $K$ on $L_2(\mathbb{R})$ is defined by
\begin{equation}\label{int oper}(Vx)(t)=\int_{\mathbb{R}}K(t,s)x(s)ds, \,\,\ x\in L_2(\mathbb{R}).
\end{equation}
Then for any $V\in \mathcal{E}(H),$ we define the operator $T(V)$ as follows (see \cite{GK1,GK2} for more details)
\begin{equation}\label{T-oper}
(T(V)x)(t)=\int_{\mathbb{R}}K(t,s){\rm sgn}(t-s)x(s)ds, \,\,\ x\in L_2(\mathbb{R}).
\end{equation}
Let $H=L_{2}(\mathbb{R}).$ The following theorem gives a weak type estimate for the operator $T.$
\begin{thm}\label{weak est for T} For all $V\in \mathcal{L}_{1}(H)$ defined by \eqref{int oper},
we have
$$\|T(V)\|_{\mathcal{L}_{1,\infty}(H)}\leq 20\|V\|_{\mathcal{L}_{1}(H)}.$$
\end{thm}
\begin{proof} Let
$$(Vx)(t)=\int_{\mathbb{R}}K(t,s)x(s)ds, \,\,\ x\in L_2(\mathbb{R})$$
such that $V\in \mathcal{L}_{1}(H).$ For any $n\in \mathbb{Z}_{+},$ define
$$\Delta_{n}=\Big([-2^{n},2^{n}]\times[-2^{n},2^{n}]\Big)\setminus\bigcup_{k=-2^{2n}}^{2^{2n}-1}\Big(\Big[\frac{k}{2^{n}},\frac{k+1}{2^{n}}\Big]\times\Big[\frac{k}{2^{n}},\frac{k+1}{2^{n}}\Big]\Big).$$
Then $\Delta_{n}\nearrow \mathbb{R}\times\mathbb{R}$ as $n\rightarrow \infty.$
Let $$(V_{n}x)(t)=\int_{\mathbb{R}}K(t,s)\chi_{\Delta_{n}}(t,s)x(s)ds, \,\,\ x\in L_2(\mathbb{R}).$$ Hence, by \eqref{T-oper}, we have
$$(T(V_{n})x)(t)=\int_{\mathbb{R}}K_{n}(t,s){\rm sgn}(t-s)x(s)ds, \,\,\ x\in L_2(\mathbb{R}),$$
where $K_{n}=K\chi_{\Delta_{n}},$ $n\in\mathbb{Z}_{+}.$
If define $P_{k}:=M_{\chi_{[\frac{k}{2^{n}},\frac{k+1}{2^{n}}]}}$ on $L_{2}(\mathbb{R}),$ then
$$T(V_{n})=\sum_{j,k=-2^{n}}^{2^{n}-1}{\rm sgn}(j-k)P_{j}V_{n}P_{k}, \quad n\in \mathbb{Z}_{+}.$$
Since $T$ is linear and self-adjoint on $\mathcal{L}_{2}(H),$ it follows that
$$\|T(V)-T(V_{n})\|_{\mathcal{L}_{2}(H)}=\|V-V_{n}\|_{\mathcal{L}_{2}(H)}\rightarrow 0\quad \text{as} \quad n\rightarrow \infty.$$
On the other hand, for each $n\in \mathbb{Z}_{+},$ we have
\begin{equation}\label{int est}
 \|V_{n}\|_{\mathcal{L}_{1}(H)}\leq2\|V\|_{\mathcal{L}_{1}(H)}, \quad V\in \mathcal{L}_{1}(H).
\end{equation}
By Theorem 1.4 in \cite{DDPS} and \eqref{int est}, we have
$$\|T(V_{n})\|_{\mathcal{L}_{1,\infty}(H)}\leq 10\|V_{n}\|_{\mathcal{L}_{1}(H)}\leq 20\|V\|_{\mathcal{L}_{1}(H)}, \quad V\in \mathcal{L}_{1}(H).$$
Since the quasi-norm in $\mathcal{L}_{1,\infty}(H)$ has the Fatou property, it follows that
$$\|T(V)\|_{\mathcal{L}_{1,\infty}(H)}\leq 20\|V\|_{\mathcal{L}_{1}(H)}, \quad V\in \mathcal{L}_{1}(H).$$
\end{proof}

\subsection{Double operator integrals}\label{doi subsection} Let $A$ be a self-adjoint operator affiliated with $\mathcal{M}$ and $\xi$ be a bounded Borel function on $\mathbb{R}^2$. Symbolically, a double operator integral is defined by the formulae
\begin{equation}\label{doi def}
T_{\xi}^{A,A}(V)=\int_{\mathbb{R}^2}\xi(\lambda,\mu)dE_A(\lambda)VE_A(\mu),\quad V\in \mathcal{L}_2(\mathcal{M}).
\end{equation}
For a more rigorous definition, consider projection valued measures on $\mathbb{R}$ acting on the Hilbert space $\mathcal{L}_2(\mathcal{M})$ by the formulae $X\to E_A(\mathcal{B})X$ and $X\to XE_A(\mathcal{B}).$ These spectral measures commute and, hence (see Theorem
V.2.6 in \cite{BirSol}), there exists a countably additive (in the strong operator topology) projection-valued measure $\nu$ on $\mathbb{R}^2$ acting on the Hilbert space $\mathcal{L}_2(\mathcal{M})$ by the formulae
$$\nu(\mathcal{B}_1\otimes\mathcal{B}_2):X\to E_A(\mathcal{B}_1)XE_A(\mathcal{B}_2),\quad X\in \mathcal{L}_2(\mathcal{M}).$$
Integrating a bounded Borel function $\xi$ on $\mathbb{R}^2$ with respect to the measure $\nu$ produces a bounded operator acting on the Hilbert space $\mathcal{L}_2(\mathcal{M}).$ In what follows, we denote the latter operator by
$T_{\xi}^{A,A}$ (see also \cite[Remark 3.1]{PSW}).

We are mostly interested in the case $\xi=f^{[1]}$ for a Lipschitz function $f:\mathbb{R}\rightarrow \mathbb{C}.$ Here,
$$f^{[1]}(\lambda,\mu)=
\begin{cases}
\frac{f(\lambda)-f(\mu)}{\lambda-\mu},\quad \lambda\neq\mu\\
0,\quad \lambda=\mu.
\end{cases}
$$

\section{Statement of the main results}\label{statement section}

In this section, we describe our main technical tools. Firstly, we emphasize the deep connection between the studies of operators $1\otimes\mathcal{H}$ and $T.$ While connection has been noted before (see \cite{Ar,DDdePS,DDPS,GK1,GK2,Ran1,Ran2}), our approach is distinct to all previous approaches. We are able to present a single abstract approach to the study of self-adjoint contractions on semifinite von Neumann algebras $(\mathcal{M},\tau)$ (see Subsection \ref{s number section}), which allows us to treat these two operators from a single perspective. In particular, we are in a position to give a precise description of optimal ranges of all just cited operators.

We need the following result.
\begin{proposition}\label{Separab} Let $\mathcal{M}$ be a semifinite  von Neumann algebra equipped with a faithful normal semifinite trace $\tau$ and $\mathcal{E}_0(\mathcal{M})$ be a symmetric (quasi-)Banach operator space.
Suppose that $\mathcal{E}_0(\mathcal{M})$ has order continuous (quasi-)norm. Then $(\mathcal{E}_0\cap \mathcal{L}_2)(\mathcal{M})$ is dense in $\mathcal{E}_0(\mathcal{M}).$
\end{proposition}
\begin{proof}
Let $\mathcal{F}^{r}(\mathcal{M},\tau)$ be the set of all $\tau$-finite range (or $\tau$-finite rank) operators in $\mathcal{M}$ (see \cite[Section 2.4, p. 210]{DdeP}) and let
$\overline{\mathcal{F}^{r}(\mathcal{M},\tau)}^{\|\cdot\|_{\mathcal{E}_{0}(\mathcal{M})}}$ be its closure in $\mathcal{E}_{0}(\mathcal{M})$. It is clear that $\mathcal{F}^{r}(\mathcal{M},\tau)\subset(\mathcal{E}_0\cap \mathcal{L}_2)(\mathcal{M})\subset \mathcal{E}_0(\mathcal{M}).$
Since the (quasi-)norm on $\mathcal{E}_0(\mathcal{M})$ is order continuous by assumption, it follows that $\overline{\mathcal{F}^{r}(\mathcal{M},\tau)}^{\|\cdot\|_{\mathcal{E}_{0}(\mathcal{M})}}=\mathcal{E}_{0}(\mathcal{M})$ (see \cite[Lemma 4.9, Chapter
IV]{DdePS}, or \cite{HS} for a more general case). This shows that $(\mathcal{E}_0\cap \mathcal{L}_2)(\mathcal{M})$ is dense in $\mathcal{E}_0(\mathcal{M}).$
\end{proof}
Now, we adopt the following convention to explain our approach.
\begin{convention}\label{main setting} Let $\mathcal{M}$ be a semifinite von Neumann algebra equipped with a faithful normal semifinite trace $\tau.$
Also assume that $\mathcal{T}:\mathcal{L}_2(\mathcal{M})\to \mathcal{L}_2(\mathcal{M})$ is a self-adjoint contraction. Let $\mathcal{E}_0(\mathcal{M})$ and $\mathcal{E}_1(\mathcal{M})$ be symmetric (quasi-)Banach operator spaces. Suppose that the norm on $\mathcal{E}_0(\mathcal{M})$ is order continuous.
We say that $\mathcal{T}:\mathcal{E}_0(\mathcal{M})\to \mathcal{E}_1(\mathcal{M})$ if
\begin{equation}\label{Contraction}
\|\mathcal{T}(V)\|_{\mathcal{E}_1(\mathcal{M})}\leq c(\mathcal{T})\|V\|_{\mathcal{E}_0(\mathcal{M})},\quad \forall V\in (\mathcal{E}_0\cap \mathcal{L}_2)(\mathcal{M}).
\end{equation}
Since $(\mathcal{E}_0\cap \mathcal{L}_2)(\mathcal{M})$ is dense in $\mathcal{E}_0(\mathcal{M})$ by Proposition \ref{Separab}, it follows that $\mathcal{T}$ admits a unique bounded linear extension $\mathcal{T}:\mathcal{E}_0(\mathcal{M})\to \mathcal{E}_1(\mathcal{M}).$
\end{convention}
Throughout this paper, we shall use the symbol $\mathcal{A}\lesssim \mathcal{B}$ to indicate that there exists a universal positive constant $c_{abs},$ independent of all important parameters, such that $\mathcal{A}\leq
c_{abs}\mathcal{B}$.
$\mathcal{A}\approx \mathcal{B}$  means that $\mathcal{A}\lesssim \mathcal{B}$ and $\mathcal{A}\gtrsim \mathcal{B}.$
 Recall that the operator $S$ is given by formulae \eqref{S}.

 The main result of the paper is the following theorem, which underpins the solution to Problems \ref{main problem} and \ref{second problem}.
\begin{thm}\label{first main theorem} Let $\mathcal{M}$ be a semifinite von Neumann algebra equipped with a faithful normal semifinite trace $\tau.$ Let $\mathcal{T}:\mathcal{L}_2(\mathcal{M})\to \mathcal{L}_2(\mathcal{M})$ be a self-adjoint
contraction.
\begin{enumerate}[{\rm (i)}]
\item Suppose that $\mathcal{T}$ admits a bounded linear extension on $\mathcal{L}_{p}(\mathcal{M})$ for all $1<p\leq 2.$ If
\begin{equation}\label{p-est}
\|\mathcal{T}\|_{\mathcal{L}_p(\mathcal{M})\to \mathcal{L}_p(\mathcal{M})}\lesssim \frac{1}{p-1},\quad 1<p\leq 2,
\end{equation} then
$$\mu(\mathcal{T}(A))\prec\prec c_{abs}S\mu(A), \,\ A\in \Lambda_{\log}(\mathcal{M}).$$
\item Suppose that $\mathcal{T}$ admits a bounded linear extension from $\mathcal{L}_{1}(\mathcal{M})$ to $\mathcal{L}_{1,\infty}(\mathcal{M}),$ that is
 \begin{equation}\label{weak est}
\|\mathcal{T}\|_{\mathcal{L}_1(\mathcal{M})\to \mathcal{L}_{1,\infty}(\mathcal{M})}\lesssim 1.
\end{equation} We have
$$\mu(\mathcal{T}(A))\leq c_{abs}S\mu(A), \,\ A\in \Lambda_{\log}(\mathcal{M}).$$
\end{enumerate}
\end{thm}
In particular, this theorem extends \cite[Theorem III.4.8, p. 138]{BSh} and \cite[Theorem III.6.8, p. 160]{BSh}.
Moreover, this theorem is also applicable for the double operator integrals (see Subsection \ref{doi subsection}) associated with  Lipschitz functions $f$ defined on $\mathbb{R}$
(see also \cite[Corollary IV.6.8, 6.9, and 6.10, p. 251]{BSh} for the applications).

\begin{rmk}\label{T dom} By Theorem \ref{weak est for T}, the operator $T$ defined in \eqref{T-oper} satisfies the assumptions of the Theorem \ref{first main theorem}. Therefore,
$T$ is dominated by the operator $S^{d}$ in the following sense
$$\mu(T(A))\leq c_{abs}S^{d}\mu(A), \quad \forall A\in \Lambda_{\log}(H),$$
where $c_{abs}$ is a positive absolute constant.
 Since, the maximal domain of $S^{d}$ is Lorentz space $\Lambda_{\log}(\mathbb{Z}_{+})$ (see \eqref{S dis}),
it follows that $T$ is defined on the Schatten-Lorentz ideal $\Lambda_{\log}(H).$
\end{rmk}

\section{An abstract operator $\mathcal{T}$ and its upper estimate}\label{statement proof}

In this section, we prove our main result Theorem \ref{first main theorem}.
The proof of 
that requires some preparation.
\begin{lem}\label{auxiliary lem1} If $X\in\bigcap_{1<p\leq2} \mathcal{L}_p(\mathcal{M})$ is such that $$\sup_{1<p\leq 2}(p-1)\|X\|_{\mathcal{L}_p(\mathcal{M})}<\infty,$$ then
$$\|X\|_{\mathcal{M}_{1,\infty}+(\mathcal{L}_1\cap \mathcal{L}_2)(\mathcal{M})}\leq c_{abs}\cdot\sup_{1<p\leq 2}(p-1)\|X\|_{\mathcal{L}_p(\mathcal{M})}.$$
\end{lem}
\begin{proof}
If $X\in S(\mathcal{M},\tau)$ is such that $\displaystyle{\sup_{1<p\leq 2}}(p-1)\|X\|_{\mathcal{L}_p(\mathcal{M})}<\infty,$ then by
$$\sum_{k=1}^{\infty}\mu^{p}(k,X)\leq \sum_{k=1}^{\infty}\int_{k-1}^{k}\mu^{p}(s,X)ds=\int_{0}^{\infty}\mu^{p}(s,X)ds,$$
we have
\begin{equation}\label{sup lp est}\begin{split}
\sup_{1<p\leq 2}(p-1)\|\mu(X)\chi_{(0,1)}\|_{L_p(0,\infty)}
+\sup_{1<p\leq 2}(p-1)\|\{\mu(k,X)\}_{k\geq 1}\|_{\ell_p(\mathbb{N})}\\
\leq 2\sup_{1<p\leq 2}(p-1)\|\mu(X)\|_{L_p(0,\infty)}=2\sup_{1<p\leq 2}(p-1)\|X\|_{\mathcal{L}_p(\mathcal{M})},
\end{split}\end{equation}
where $\ell_p(\mathbb{N})$ is the space of all $p$-summable sequences (see \cite[Chapter II, p. 53]{LT}).
Let $X\in S(\mathcal{M},\tau)$ such that $\|\mu(X)\chi_{(0,1)}\|_{L_2(0,\infty)}<\infty.$
Since $L_2(0,1)\subseteq L_p(0,1)$ for $p\in(1,2],$ it follows that
$$\|\mu(X)\chi_{(0,1)}\|_{L_p(0,\infty)}
\leq\|\mu(X)\chi_{(0,1)}\|_{L_2(0,\infty)}, \,\ 1<p\leq 2.$$
Taking supremum, over $1<p\leq 2,$ we obtain
\begin{equation}\label{l2approx}
\sup_{1<p\leq 2}(p-1)\|\mu(X)\chi_{(0,1)}\|_{L_p(0,\infty)}=\|\mu(X)\chi_{(0,1)}\|_{L_2(0,\infty)}.
\end{equation}
On the other hand, by Theorem 4.5 in \cite{CRSS}, we have
\begin{equation}\label{CRSS}
\sup_{1<p\leq 2}(p-1)\|\{\mu(k,X)\}_{k\geq 1}\|_{\ell_p(\mathbb{N})}\approx\|\{\mu(k,X)\}_{k\geq 1}\|_{M_{1,\infty}(\mathbb{N})},
\end{equation}
(see \eqref{Marsecuence}).
Combining \eqref{sup lp est}, \eqref{l2approx}, and  \eqref{CRSS}, we obtain
\begin{eqnarray*}\begin{split}\|X\|_{\mathcal{M}_{1,\infty}+(\mathcal{L}_1\cap \mathcal{L}_2)(\mathcal{M})}
&\lesssim\|\mu(X)\chi_{(0,1)}\|_{L_2(0,\infty)}+\|\mu(X)\chi_{(1,\infty)}\|_{\mathcal{M}_{1,\infty}(\mathcal{M})}\\
&\approx\|\mu(X)\chi_{(0,1)}\|_{L_2(0,\infty)}+\|\{\mu(k,X)\}_{k\geq 1}\|_{M_{1,\infty}(\mathbb{N})}\\
&\approx\sup_{1<p\leq 2}(p-1)\|\mu(X)\chi_{(0,1)}\|_{L_p(0,\infty)}\\
&+\sup_{1<p\leq 2}(p-1)\|\{\mu(k,X)\}_{k\geq 1}\|_{\ell_p(\mathbb{N})}\\
&\approx\sup_{1<p\leq 2}(p-1)\|X\|_{\mathcal{L}_p(\mathcal{M})}.
\end{split}\end{eqnarray*}
\end{proof}
\begin{lem}\label{auxiliary lem2} The following
$$\Big(\mathcal{M}_{1,\infty}+(\mathcal{L}_1\cap \mathcal{L}_2)(\mathcal{M})\Big)^{\times}=\Lambda_{\log}\cap(\mathcal{L}_2+ \mathcal{L}_\infty)(\mathcal{M})$$
is isometric.
\end{lem}
\begin{proof}
Since
$$\Big(\mathcal{M}_{1,\infty}+(\mathcal{L}_1\cap \mathcal{L}_2)(\mathcal{M})\Big)^{\times}=\mathcal{M}_{1,\infty}^{\times}(\mathcal{M})\cap(\mathcal{L}_1\cap \mathcal{L}_2)^{\times}(\mathcal{M})$$ and $$(\mathcal{L}_1\cap
\mathcal{L}_2)^{\times}(\mathcal{M})=\mathcal{L}_1^{\times}(\mathcal{M})+\mathcal{L}_{2}^{\times}(\mathcal{M})$$ by \eqref{sum and inter dual}, it follows from \cite[Example IV. 3.13 (a),(b), p. 32]{DdePS} and \cite[Theorem II.5.4]{KPS} (see also
\cite[Theorem 2.6.14]{LSZ}) that
$$\Big(\mathcal{M}_{1,\infty}+(\mathcal{L}_1\cap \mathcal{L}_2)(\mathcal{M})\Big)^{\times}=\mathcal{M}_{1,\infty}^{\times}(\mathcal{M})\cap(\mathcal{L}_1\cap \mathcal{L}_2)^{\times}(\mathcal{M})$$
$$=\mathcal{M}_{1,\infty}^{\times}(\mathcal{M})\cap \Big(\mathcal{L}^{\times}_1(\mathcal{M})+ \mathcal{L}_2^{\times}(\mathcal{M})\Big)=\Lambda_{\log}\cap(\mathcal{L}_2+ \mathcal{L}_\infty)(\mathcal{M}).$$
\end{proof}

\begin{lem}\label{tail lemma}
Let $\mathcal{M}$ be a von Neumann algebra which satisfies the assumption in Theorem \ref{first main theorem} and let $\mathcal{T}$ satisfy the assumption in Theorem \ref{first main theorem} (i).
Then for each $A\in \mathcal{L}_2(\mathcal{M}),$ we have
$$\|\mathcal{T}(A)\|_{(\mathcal{L}_2+\mathcal{L}_{\infty})(\mathcal{M})}\leq c_{abs}\|A\|_{\Lambda_{\log}\cap(\mathcal{L}_2+\mathcal{L}_{\infty})(\mathcal{M})}.$$
\end{lem}
\begin{proof} We split the argument into several steps.
It is easy to see that $\mathcal{L}_2(\mathcal{M})\subset\Lambda_{\log}\cap(\mathcal{L}_2+\mathcal{L}_{\infty})(\mathcal{M}).$
Indeed, for any $X\in \mathcal{L}_2(\mathcal{M})$ using Cauchy-Schwarz inequality, we obtain
\begin{eqnarray*}\begin{split}\|X\|_{\Lambda_{\log}\cap(\mathcal{L}_2+\mathcal{L}_{\infty})(\mathcal{M})}
&=\max\{\|X\|_{\Lambda_{\log}(\mathcal{M})},\|X\|_{(\mathcal{L}_2+\mathcal{L}_{\infty})(\mathcal{M})}\}\\
&\leq\|X\|_{\Lambda_{\log}(\mathcal{M})}+\|X\|_{(\mathcal{L}_2+\mathcal{L}_{\infty})(\mathcal{M})}\\
&\leq\int_{0}^{\infty}\mu(s,X)\frac{1}{1+s}ds+\|X\|_{\mathcal{L}_2(\mathcal{M})}\leq2\|X\|_{\mathcal{L}_2(\mathcal{M})}.
\end{split}\end{eqnarray*}

{\bf Step 1.} Let $Y\in (\mathcal{L}_1\cap \mathcal{L}_2)(\mathcal{M}).$ By assumption \eqref{p-est}, we have
$$(p-1)\|\mathcal{T}(Y)\|_{\mathcal{L}_p(\mathcal{M})}\lesssim\|Y\|_{\mathcal{L}_p(\mathcal{M})},\quad 1<p\leq 2.$$
Thus, taking supremum over $p\in(1,2],$ we obtain
\begin{equation}\label{sup-est}\sup_{1<p\leq 2}(p-1)\|\mathcal{T}(Y)\|_{\mathcal{L}_p(\mathcal{M})}\lesssim\sup_{1<p\leq 2}\|Y\|_{\mathcal{L}_p(\mathcal{M})}\leq c_{abs}\|Y\|_{(\mathcal{L}_1\cap \mathcal{L}_2)(\mathcal{M})}.
\end{equation}
Combining \eqref{sup-est} and Lemma \ref{auxiliary lem1}, we obtain
\begin{equation}\label{step1}
\|\mathcal{T}(Y)\|_{\mathcal{M}_{1,\infty}+(\mathcal{L}_1\cap \mathcal{L}_2)(\mathcal{M})}\leq c_{abs}\|Y\|_{(\mathcal{L}_1\cap \mathcal{L}_2)(\mathcal{M})}, \,\,\ \ Y\in(\mathcal{L}_1\cap \mathcal{L}_2)(\mathcal{M}).
\end{equation}

{\bf Step 2.} Now let $A\in \mathcal{L}_2(\mathcal{M}).$
By Definition \ref{duality}, we have
$$\|X\|_{(\mathcal{L}_2+\mathcal{L}_{\infty})(\mathcal{M})}=\sup_{\|Y\|_{(\mathcal{L}_1\cap \mathcal{L}_2)(\mathcal{M})}\leq 1}|\tau(XY^*)|, \,\ X\in (\mathcal{L}_2+\mathcal{L}_{\infty})(\mathcal{M}).
$$
Thus, for any $A\in \mathcal{L}_2(\mathcal{M}),$ we obtain
$$\|\mathcal{T}(A)\|_{(\mathcal{L}_2+\mathcal{L}_{\infty})(\mathcal{M})}=\sup_{\|Y\|_{(\mathcal{L}_1\cap \mathcal{L}_2)(\mathcal{M})}\leq 1}|\tau(\mathcal{T}(A)Y^*)|.$$
Since $A,Y\in \mathcal{L}_2(\mathcal{M})$ and by assumption $\mathcal{T}$ is self-adjoint in $\mathcal{L}_{2}(\mathcal{M})$, it follows that
\begin{equation}\label{T-eq}
\|\mathcal{T}(A)\|_{(\mathcal{L}_2+\mathcal{L}_{\infty})(\mathcal{M})}=\sup_{\|Y\|_{(\mathcal{L}_1\cap \mathcal{L}_2)(\mathcal{M})}\leq 1}|\tau(A(\mathcal{T}(Y))^*)|.
\end{equation}

So, by Lemma \ref{auxiliary lem2} and
H\"older's inequality (see Proposition \ref{Holder}),
we have
$$|\tau(A\mathcal{T}(Y)^*)|\leq c_{abs}\|A\|_{\Lambda_{\log}\cap(\mathcal{L}_2+\mathcal{L}_{\infty})(\mathcal{M})}\|\mathcal{T}(Y)\|_{\mathcal{M}_{1,\infty}+(\mathcal{L}_1\cap \mathcal{L}_2)(\mathcal{M})}.$$
By \eqref{step1}, we obtain
\begin{equation}\label{T-est}
|\tau(A(\mathcal{T}(Y))^*)|\leq c_{abs}\|A\|_{\Lambda_{\log}\cap(\mathcal{L}_2+\mathcal{L}_{\infty})(\mathcal{M})}\|Y\|_{(\mathcal{L}_1\cap \mathcal{L}_2)(\mathcal{M})}.
\end{equation}
Thus, taking supremum in \eqref{T-est} over all $Y\in(\mathcal{L}_1\cap \mathcal{L}_{2})(\mathcal{M})$ and using \eqref{T-eq}, we obtain for each $A\in \mathcal{L}_2(\mathcal{M})$
$$\|\mathcal{T}(A)\|_{(\mathcal{L}_2+\mathcal{L}_{\infty})(\mathcal{M})}\leq c_{abs}\cdot\sup_{\|Y\|_{(\mathcal{L}_1\cap \mathcal{L}_2)(\mathcal{M})}\leq 1}\|A\|_{\Lambda_{\log}\cap(\mathcal{L}_2+\mathcal{L}_{\infty})(\mathcal{M})}\|Y\|_{(\mathcal{L}_1\cap \mathcal{L}_2)(\mathcal{M})}.$$
In other words, we have
$$\|\mathcal{T}(A)\|_{(\mathcal{L}_2+\mathcal{L}_{\infty})(\mathcal{M})}\leq c_{abs}\|A\|_{\Lambda_{\log}\cap(\mathcal{L}_2+\mathcal{L}_{\infty})(\mathcal{M})}, \,\ \forall A\in \mathcal{L}_2(\mathcal{M}).$$
\end{proof}

\begin{lem}\label{main lorentz lemma}

Let the assumptions of Lemma \ref{tail lemma} hold.
We have
$$\|\mathcal{T}(A)\|_{(\mathcal{L}_1+\mathcal{L}_{\infty})(\mathcal{M})}\leq c_{abs}\|A\|_{\Lambda_{\psi}(\mathcal{M})}, \,\,\  A\in\Lambda_{\psi}(\mathcal{M}),$$
where
$$\psi(t)=\left\{
                                                                                               \begin{array}{ll}
                                                                                                 t\log(\frac{e^2}{t}), & 0<t\leq 1\hbox{,} \\
                                                                                                 2\log(et), & 1\leq t<\infty\hbox{.}
                                                                                               \end{array}
                                                                                             \right.$$
\end{lem}
\begin{proof} The proof will be divided into several steps.

{\bf Step 1.} Suppose first that $A$ is a projection and let $\tau(A)=t\in[0,1].$ We claim that
\begin{equation}\label{claim step1}
\|\mathcal{T}(A)\|_{(\mathcal{L}_1+\mathcal{L}_{\infty})(\mathcal{M})}\leq e\psi(t).
\end{equation}
By \eqref{L+L}, we have
$$
\|X\|_{(\mathcal{L}_1+\mathcal{L}_{\infty})(\mathcal{M})}\leq\inf_{\epsilon\in(0,1)}\|X\|_{\mathcal{L}_{1+\epsilon}(\mathcal{M})}, \,\,  X\in \mathcal{L}_{1+\epsilon}(\mathcal{M}).
$$
Thus,
\begin{equation}\label{inf-est}
\|\mathcal{T}(A)\|_{(\mathcal{L}_1+\mathcal{L}_{\infty})(\mathcal{M})}\leq\inf_{\epsilon\in(0,1)}\|\mathcal{T}(A)\|_{\mathcal{L}_{1+\epsilon}(\mathcal{M})}.
\end{equation}
Applying \eqref{p-est} and recalling that $\|A\|_{\mathcal{L}_{1+\varepsilon}(\mathcal{M})}=t^{\frac1{1+\epsilon}}$, we obtain
$$\|\mathcal{T}(A)\|_{\mathcal{L}_{1+\epsilon}(\mathcal{M})}\leq\epsilon^{-1}\|A\|_{\mathcal{L}_{1+\epsilon}(\mathcal{M})}=\epsilon^{-1}t^{\frac1{1+\epsilon}}.$$
Hence, taking infimum over all $\varepsilon\in(0,1)$ from the preceding inequality and by \eqref{inf-est}, we have
\begin{equation}\label{Teps-est}
\|\mathcal{T}(A)\|_{(\mathcal{L}_1+\mathcal{L}_{\infty})(\mathcal{M})}\leq\inf_{\epsilon\in(0,1)}\epsilon^{-1}t^{\frac1{1+\epsilon}}.
\end{equation}
For $t<\frac1e,$ set $\epsilon=\frac1{\log(\frac1t)}.$
We have
\begin{equation}\label{comb1}
\|\mathcal{T}(A)\|_{(\mathcal{L}_1+\mathcal{L}_{\infty})(\mathcal{M})}\leq t\log\left(\frac1t\right)\cdot \left(\frac1t\right)^{\frac{\epsilon}{1+\epsilon}}=t\log\left(\frac1t\right)\cdot e^{\frac1{1+\epsilon}}\leq et\log\left(\frac1t\right)\leq e\psi(t).
\end{equation}
If $t\in[\frac1e,1],$ then setting $\epsilon=1,$ from \eqref{Teps-est} we obtain
\begin{equation}\label{comb2}
\|\mathcal{T}(A)\|_{(\mathcal{L}_1+\mathcal{L}_{\infty})(\mathcal{M})}\leq t^{\frac12}\leq 1\leq e\psi(t).
\end{equation}
A combination of the \eqref{comb1} and \eqref{comb2} establishes the claim \eqref{claim step1} of Step 1.

{\bf Step 2.} Suppose now that $A$ is a projection and let $\tau(A)=t\in[1,\infty).$ We claim that
\begin{equation}\label{claim step2}
\|\mathcal{T}(A)\|_{(\mathcal{L}_1+\mathcal{L}_{\infty})(\mathcal{M})}\leq c_{abs}\psi(t).
\end{equation}
Since $\mathcal{L}_{2}(\mathcal{M}),\mathcal{L}_{\infty}(\mathcal{M})\subset(\mathcal{L}_1+\mathcal{L}_{\infty})(\mathcal{M})$ by \eqref{imbed},
 it follows that
\begin{equation}\label{x-est}
\|X\|_{(\mathcal{L}_1+\mathcal{L}_{\infty})(\mathcal{M})}\leq c_{abs}\|X\|_{(\mathcal{L}_2+\mathcal{L}_{\infty})(\mathcal{M})}, \,\, X\in (\mathcal{L}_2+\mathcal{L}_{\infty})(\mathcal{M}).
\end{equation}
By \eqref{x-est} and Lemma \ref{tail lemma}, we have
$$\|\mathcal{T}(A)\|_{(\mathcal{L}_1+\mathcal{L}_{\infty})(\mathcal{M})}\leq c_{abs}\|A\|_{\Lambda_{\log}(\mathcal{M})\cap(\mathcal{L}_2+\mathcal{L}_{\infty})(\mathcal{M})}.$$
Since $A$ is a projection with $\tau(A)>1,$ it follows from the preceding inequality that
$$\|\mathcal{T}(A)\|_{(\mathcal{L}_1+\mathcal{L}_{\infty})(\mathcal{M})}\leq c_{abs}\log(1+t)\leq c_{abs}\psi(t).$$

{\bf Step 3.} Let $A$ be a positive operator of the form

$$
A=\sum_{k=1}^n\alpha^{'}_kP^{'}_k,
$$
where, $\alpha^{'}_k\in (0,\infty)$ and the $P^{'}_k$ are pairwise orthogonal projections with finite trace. Rearranging the summation, we may assume that $\{\alpha^{'}_k\}_{k=1}^{n}$ is increasing. Let
$\alpha_k=\alpha^{'}_k-\alpha^{'}_{k+1}$ and $P_k=P_1^{'}+\cdots+P_k^{'}$
for $1\leq k\leq n$ with $\alpha^{'}_{n+1}=0.$ Then $\{P_k\}_{k=1}^{n}$ is an increasing sequence of projections and
\begin{equation}\label{simple form}
A=\sum_{k=1}^n\alpha_kP_k.
\end{equation}
Since $\mu(A)=\sum_{k=1}^n\alpha_k\chi_{[0,\tau(P_k))}$ (see \cite[Example III. 2.2 (i), p. 10]{DdePS}), it follows from \eqref{Lphi} that
\begin{equation}\label{xi-eq}
\|A\|_{\Lambda_{\psi}(\mathcal{M})}=\sum_{k=1}^n\alpha_k\psi(\tau(P_k)).
\end{equation}
On the other hand, by the triangle inequality, we have
\begin{equation}\label{sum est}
\|\mathcal{T}(A)\|_{(\mathcal{L}_1+\mathcal{L}_{\infty})(\mathcal{M})}\leq\sum_{k=1}^n\alpha_k\|\mathcal{T}(P_k)\|_{(\mathcal{L}_1+\mathcal{L}_{\infty})(\mathcal{M})}.
\end{equation}
By \eqref{claim step1} and \eqref{claim step2}, we have
\begin{equation}\label{pk-est}\|\mathcal{T}(P_k)\|_{(\mathcal{L}_1+\mathcal{L}_{\infty})(\mathcal{M})}\leq c_{abs}\psi(\tau(P_k)).
\end{equation}
Thus, combining \eqref{xi-eq}, \eqref{sum est}, and \eqref{pk-est}, we obtain
\begin{equation}\label{step-3}\|\mathcal{T}(A)\|_{(\mathcal{L}_1+\mathcal{L}_{\infty})(\mathcal{M})}\leq c_{abs}\cdot\sum_{k=1}^n\alpha_k\psi(\tau(P_k))=c_{abs}\|A\|_{\Lambda_{\psi}(\mathcal{M})}.
\end{equation}

{\bf Step 4.} As in the proof of Proposition \ref{Separab}, let $\mathcal{F}^{r}(\mathcal{M},\tau)$ be the set of all $\tau$-finite range operators in $\mathcal{M}.$ The spectral theorem guaranties that the set of positive operators
having the form \eqref{simple form} is dense in $\mathcal{F}^{r}(\mathcal{M},\tau)$ (see the proof of \cite[Lemma IV. 6.8, p. 65]{DdePS}). Since $\Lambda_{\psi}(\mathcal{M})$ is a separable Banach space, it follows from \cite[Lemma IV.8.5, p. 87]{DdePS} (see also \cite[Theorem 55]{DdeP}) that $\mathcal{F}^{r}(\mathcal{M},\tau)$ is dense in $\Lambda_{\psi}(\mathcal{M}).$ So, it follows that the set of positive
operators having the form \eqref{simple form} is dense in $\Lambda_{\psi}(\mathcal{M}).$
Let $A\in\Lambda_{\psi}(\mathcal{M})$ and let $\{A_n\}_{n=1}^{\infty}$ be a sequence of positive operators of the form given by \eqref{simple form} and approximating $A$ in the norm $\|\cdot\|_{\Lambda_{\psi}(\mathcal{M})}.$
Then by \eqref{step-3}, we have $\|\mathcal{T}(A_n-A_m)\|_{(\mathcal{L}_1+\mathcal{L}_{\infty})(\mathcal{M})}\leq c_{abs}\|A_n-A_m\|_{\Lambda_{\psi}(\mathcal{M})},$ and so the sequence $\{\mathcal{T}(A_n)\}_{n\geq 1}$ is Cauchy in
$(\mathcal{L}_1+\mathcal{L}_{\infty})(\mathcal{M}).$ Since the space $(\mathcal{L}_1+\mathcal{L}_{\infty})(\mathcal{M})$ is complete (see \cite[Chapter III, p. 98]{DdePS} for more details), it follows that $\mathcal{T}(A_n)$ converges to an element of
$(\mathcal{L}_1+\mathcal{L}_{\infty})(\mathcal{M}),$ and we denote the limit by $\mathcal{T}(A).$  Thus, again using \eqref{step-3}, we obtain
\begin{eqnarray*}\begin{split}\|\mathcal{T}(A)\|_{(\mathcal{L}_1+\mathcal{L}_{\infty})(\mathcal{M})}&=\lim_{n\rightarrow \infty}\|\mathcal{T}(A_n)\|_{(\mathcal{L}_1+\mathcal{L}_{\infty})(\mathcal{M})}\\
&\leq c_{abs}\lim_{n\rightarrow \infty}\|A_n\|_{\Lambda_{\psi}(\mathcal{M})}=c_{abs}\|A\|_{\Lambda_{\psi}(\mathcal{M})}, \,\,\ A\in\Lambda_{\psi}(\mathcal{M}).
\end{split}\end{eqnarray*}
Therefore, the proof is complete.
\end{proof}

We are now fully equipped to prove the first part of our main result.
\begin{proof}[Proof of Theorem \ref{first main theorem} (i)] By Lemma \ref{main lorentz lemma}, we have
$$\|\mathcal{T}(A)\|_{(\mathcal{L}_1+\mathcal{L}_{\infty})(\mathcal{M})}\leq c_{abs}\|A\|_{\Lambda_{\psi}(\mathcal{M})}, \,\, A\in\Lambda_{\psi}(\mathcal{M}).$$
By \eqref{L+L} and \eqref{Lphi}, we obtain
$$\|A\|_{\Lambda_{\psi}(\mathcal{M})}\approx\|S\mu(A)\|_{(\mathcal{L}_1+\mathcal{L}_{\infty})(\mathcal{M})}, \,\, A\in\Lambda_{\psi}(\mathcal{M}).$$
Thus, again using \eqref{L+L}, we infer from the preceding estimate

\begin{equation}\label{*}
\int_0^1\mu(s,\mathcal{T}(A))ds\leq c_{abs}\cdot\int_0^1(S\mu(A))(s)ds, \,\, A\in\Lambda_{\psi}(\mathcal{M}).
\end{equation}
Now, we scale the trace $\tau\rightarrow t^{-1}\tau.$ We have $\mu_{t^{-1}\tau}(s,X)=\mu_{\tau}(st,X).$ Note that,
$$\|\mathcal{T}(A)\|_{\mathcal{L}_{p}(\mathcal{M},\frac{\tau}{t})}\lesssim \frac{1}{p-1}\|A\|_{\mathcal{L}_{p}(\mathcal{M},\frac{\tau}{t})}, 1<p\leq2.$$
Hence, applying \eqref{*} to $(\mathcal{M},\frac{\tau}{t}),$ we obtain
$$\int_0^1\mu_{\tau}(st,\mathcal{T}(A))ds\leq c_{abs}\cdot\int_0^1(S\mu_{\tau}(A))(st)ds, \,\, A\in\Lambda_{\psi}(\mathcal{M}).$$
Therefore, we have
\begin{eqnarray*}\begin{split}
\frac1t\int_0^t\mu(s,\mathcal{T}(A))ds
&=\int_0^1\mu(st,\mathcal{T}(A))ds\\
&\leq c_{abs}\cdot\int_0^1(S\mu(A))(st)ds=c_{abs}\cdot\frac1t\int_0^t(S\mu(A))(s)ds.
\end{split}\end{eqnarray*}
Since $t>0$ is arbitrary, the assertion follows.
\end{proof}

To prove Theorem \ref{first main theorem} (ii), we need the following lemma.
\begin{lem}\label{S by S} Let $A\in \Lambda_{\log}(\mathcal{M}).$ We have,
$$(S\mu(A))(t)\leq4(S\mu(A))(2t), \,\ t>0.$$
\end{lem}
\begin{proof} If $A\in \Lambda_{\log}(\mathcal{M}),$ then
\begin{eqnarray*}\begin{split}
&(S\mu(A))(t)\stackrel{\eqref{S}}{=}\frac{1}{t}\int_{0}^{t}\mu(s,A)ds+\int_t^{\infty}\mu(s,A)\frac{ds}{s}\\
&\leq \frac{1}{t}\int_{0}^{t}\mu(s,A)ds+\int_{2t}^{\infty}\mu(s,A)\frac{ds}{s}+\int_{t}^{2t}\mu(s,A)\frac{ds}{s}\\
&\leq \frac{2}{t}\int_{0}^{t}\mu(s,A)ds+\int_{2t}^{\infty}\mu(s,A)\frac{ds}{s}\\
&\leq 4\left(\frac{1}{2t}\int_{0}^{2t}\mu(s,A)ds+\int_{2t}^{\infty}\mu(s,A)\frac{ds}{s}\right)=4(S\mu(A))(2t), \,\ t>0.
\end{split}\end{eqnarray*}
\end{proof}

We are now ready to prove second part of the Theorem \ref{first main theorem}.
\begin{proof}[Proof of Theorem \ref{first main theorem} (ii)] By complex interpolation (see for instance \cite[Theorem 4.8]{Di}),
we have
$$\|\mathcal{T}\|_{\mathcal{L}_p(\mathcal{M})\to \mathcal{L}_p(\mathcal{M})}\lesssim\frac1{p-1},\quad 1<p\leq 2.$$
Thus, $\mathcal{T}$ satisfies the assumptions (i) in Theorem \ref{first main theorem}.

First we prove the assertion for positive elements from $\Lambda_{\log}(\mathcal{M}).$ Let $A\in \Lambda_{\log}(\mathcal{M})$ be a positive operator.
Fix $t>0$ and set
$$A_{1}=(A-\mu(t,A))_{+}, \,\, A_{2}=\min\{A,\mu(t,A)\}.$$
Then, by Lemma 2.5 (iv) in \cite{FK}, we have
\begin{equation}\label{BC}
 \mu(A_1)=(\mu(A)-\mu(t,A))_+,  \,\,\ \mu(A_2)=\min\{\mu(A),\mu(t,A)\}.
\end{equation}
By \eqref{triangle svf}, we obtain
$$\mu(2t,\mathcal{T}(A))\leq\mu(t,\mathcal{T}(A_1))+\mu(t,\mathcal{T}(A_2)).$$
Then, it follows from \eqref{weakL1} and assumption \eqref{weak est} that
$$t\mu(t,\mathcal{T}(A_1))\leq\|\mathcal{T}(A_1)\|_{\mathcal{L}_{1,\infty}(\mathcal{M})}\lesssim\|A_1\|_{\mathcal{L}_{1}(\mathcal{M})}=\int_0^t(\mu(s,A)-\mu(t,A))ds.$$
Hence, dividing by $t,$ we have
\begin{equation}\label{T(B)}
\mu(t,\mathcal{T}(A_1))\lesssim\frac{1}{t}\int_0^t\left(\mu(s,A)-\mu(t,A)\right)ds\lesssim(S\mu(A))(t).
\end{equation}
Since $\mu$ is decreasing function (see \cite[Chapter II, p. 59]{KPS}), it follows that
$$t\mu(t,\mathcal{T}(A_2))\leq\int_0^t\mu(s,\mathcal{T}(A_2))ds.$$
Applying Theorem \ref{first main theorem} (i) to the last inequality, we infer

\begin{equation} \label{T(C)}
t\mu(t,\mathcal{T}(A_2))\leq c_{abs}\cdot\int_0^t(S\mu(A_2))(s)ds.
\end{equation}
We now compute
\begin{eqnarray*}\begin{split}\int_0^t(S\mu(A_2))(s)ds
&\stackrel{\eqref{S}}{=}\int_0^t\Big(\frac1s\int_0^s\mu(u,A_2)du+\int_s^{\infty}\mu(u,A_2)\frac{du}{u}\Big)ds\\
&=\int_0^t\Big(\frac1s\int_0^s\mu(u,A_2)du+\int_s^t\mu(u,A_2)\frac{du}{u}+\int_t^{\infty}\mu(u,A_2)\frac{du}{u}\Big)ds.
\end{split}\end{eqnarray*}
It is clear from \eqref{BC} that
$$\mu(u,A_2)=\left\{
    \begin{array}{ll}
      \mu(t,A), & \,\,\ 0<u\leq t\hbox{,} \\
      \mu(u,A), & \,\,\ t\leq u<\infty \hbox{.}
    \end{array}
  \right.$$
Therefore,
\begin{eqnarray*}\begin{split}\int_0^t(S\mu(A_2))(s)ds
&=\int_0^t\Big(\mu(t,A)+\mu(t,A)\cdot\int_s^t\frac{du}{u}+\int_t^{\infty}\mu(u,A)\frac{du}{u}\Big)ds\\
&=t\mu(t,A)+\mu(t,A)\cdot\int_0^t\log\left(\frac{t}{s}\right)ds+t\int_t^{\infty}\mu(u,A)\frac{du}{u}\\
&=2t\mu(t,A)+t\int_t^{\infty}\mu(u,A)\frac{du}{u}\leq2t(S\mu(A))(t), \,\, t>0.
\end{split}\end{eqnarray*}
Thus, by \eqref{T(C)}

$$\mu(t,\mathcal{T}(A_2))\leq c_{abs}(S\mu(A))(t), \,\, t>0.$$
Combining \eqref{T(B)} and the preceding estimate, we obtain
\begin{equation}\label{A2 est}
\mu(2t,\mathcal{T}(A))\leq c_{abs}(S\mu(A))(t), \,\, t>0.
\end{equation}
Using \eqref{A2 est} and Lemma \ref{S by S}, we obtain
$$\mu(2t,\mathcal{T}(A))\leq c_{abs}(S\mu(A))(2t), \,\ t>0.$$
Since $t>0$ is arbitrary, the assertion follows.

Let us prove the assertion for the general case.
Note that, every operator $A$ in $S(\mathcal{M},\tau)$ can be decomposed into self-adjoint components $A=\Re(A)+i\Im(A)$  and
$\Re(A)=\frac{1}{2}(A+A^{*})$ and $\Im(A)=\frac{1}{2i}(A-A^{*})$ (see \cite[Chapter II., pp. 14-15]{DdePS}). Every self-adjoint operator $A=A^{*}$ decomposes into positive components
$$A=A_{+}-A_{-},$$
where $A_{+}=\frac{1}{2}(A+|A|)$  and $A_{-}=\frac{1}{2}(A-|A|), \,\ (|A|^{2}:=A^{*}A).$
Thus, any operator $A$ in $S(\mathcal{M},\tau)$ is represented as a linear combination of four positive operators, i.e. $A=A_1-A_2+iA_3-iA_4$ (see also \cite[Chapter I., p. 27]{LSZ}).
Note that $\mu(A_{k})\leq \sigma_{2}\mu(A),$ $k=1,2,3,4.$ So, using linearity of the operator $\mathcal{T},$ applying equation \eqref{triangle svf},
 we have
\begin{eqnarray*}\begin{split}\mu(\mathcal{T}(A))(t)
&\leq\mu\Big(\mathcal{T}(A_{1})-\mathcal{T}(A_{2})\Big)\Big(\frac{t}{2}\Big)+\mu\Big(\mathcal{T}(A_{3})-\mathcal{T}(A_{4})\Big)\Big(\frac{t}{2}\Big)\\
&\leq\sum_{k=1}^{4}\mu\Big(\mathcal{T}(A_{k})\Big)\Big(\frac{t}{4}\Big)\leq \sum_{k=1}^{4}\Big(S\mu(A_{k})\Big)\Big(\frac{t}{4}\Big)\\
&\leq c_{abs}\Big(S\mu(A)\Big)\Big(\frac{t}{8}\Big)\leq c_{abs}(S\mu(A))(t), \,\ t>0.
\end{split}\end{eqnarray*}
Since it is hold for any $t>0$, this concludes the proof.
\end{proof}

\section{Lower estimate for the triangular truncation operator $T$}

Let $S^{d}$ be the discrete version of the operator $S$ defined in \eqref{S dis}.
We will denote by $\mathbb{T}$ the circle, i.e. $\mathbb{T}=\{e^{i\theta}: \theta\in \mathbb{R}\}.$
There is an obvious identification between functions on $\mathbb{T}$ and $2\pi$-periodic functions on $\mathbb{R}$  (see \cite[Chapter I]{K}).
We identify $L_2(-\pi,\pi)$ with $L_2(\mathbb{T}),$ where $L_2(\mathbb{T})$ is the Lebesgue space of (equivalence classes) measurable functions such that
$$\|f\|_{L_{2}(\mathbb{T})}:=\left(\frac{1}{2\pi}\int_{\mathbb{T}}|f(t)|^{2}dt\right)^{1/2}$$
is finite. For more details on Fourier analysis on $\mathbb{T}$, we refer to \cite[Chapter I]{K}.

Let $T$ be the operator defined by \eqref{T-oper} and we will denote by $T_{[-\pi,\pi]}$ the operator defined by
\begin{equation}\label{T-oper on the interval}(T_{[-\pi,\pi]}(V)x)(t)=\int_{-\pi}^{\pi}K(t,s){\rm sgn}(t-s)x(s)ds,\quad x\in L_2(-\pi,\pi).
\end{equation}
Here,
$$(Vx)(t)=\int_{-\pi}^{\pi}K(t,s)x(s)ds,\quad x\in L_2(-\pi,\pi).$$
Let $U: L_{2}(\mathbb{R})\rightarrow L_{2}(-\pi,\pi)$ be a unitary operator is defined by
\begin{equation}\label{unitary op}
(Ux)(t):=\frac{1}{\sqrt{2}}x\Big(\tan\Big(\frac{t}{2}\Big)\Big)\cdot\frac{1}{\cos\Big(\frac{t}{2}\Big)}, \quad t\in(-\pi,\pi).
\end{equation}
If we define $R$ by the formulae
$$R(V)=UVU^{-1},$$ then
\begin{equation}\label{rest of T oper}T_{[-\pi,\pi]}\circ R=R\circ T.
\end{equation}

We denote by $\mathcal{D}$ the differential operator $\mathcal{D}:=\frac{1}{i}\frac{d}{dt}$ defined
on the set of all absolutely continuous functions $f$ on $[-\pi,\pi]$ such that $f'\in L_{2}(-\pi,\pi)$ and $f(-\pi)=f(\pi).$
For more details on this differential operator we refer the reader to \cite[Chapter VIII, pp. 275-285 ]{RS}.

Define the operator (we will denote it by $\mathcal{H}_{d}$) as follows
\begin{equation}\label{H dis}(\mathcal{H}_{d}a)(n):=\frac{2}{\pi i}\sum_{\substack{k\in\mathbb{Z}\\ k=n+1mod2}}\frac{a(k)}{k-n}, \,\,\ a\in \ell_{\infty}(\mathbb{Z}).
\end{equation}
The following theorem gives a non-commutative analogue of the Proposition 4.10 in \cite[Chapter III., p. 140]{BSh}.
\begin{thm} \label{T theorem}For every $a\in \ell_{\infty}(\mathbb{Z}),$ there exists an operator $V$ on $L_{2}(-\pi,\pi)$ with $\mu(V)=\mu(a)$ such that
$$S^{d}\mu(V)\leq c_{abs}\mu(T(V)).$$
\end{thm}
The proof needs some preparation.

\begin{lem}\label{first lem for T} For any $a\in\ell_{\infty}(\mathbb{Z}),$ we have
$$T_{[-\pi,\pi]}(a(\mathcal{D}))=(\mathcal{H}_{d}a)(\mathcal{D}).$$
\end{lem}
\begin{proof}
Take $a\in \ell_{\infty}(\mathbb{Z}),$ and consider an operator $V=a(\mathcal{D})$ (see Subsection 2.8) on $L_2(-\pi,\pi).$ It is easy to see by functional calculus that
\begin{equation}\label{ak}
a(\mathcal{D})e_k:=a(k)e_k, \,\,\  k\in\mathbb{Z},
\end{equation}
where $\{e_k(t)=e^{ikt}\}_{k=-\infty}^{+\infty}$ is complete
orthonormal system in $L_2(-\pi,\pi).$

It is well known that
$$x=\lim_{N\rightarrow \infty}\sum_{-N}^{N}\widehat{x}(n)e^{int}$$
in the $L_2(-\pi,\pi)$ norm (see \cite[Theorem I. 5.5, pp. 29-30]{K}), where $\widehat{x}(n)$ $(n\in\mathbb{Z}),$ is the n'th Fourier coefficient of the function $x$ defined by
\begin{equation}\label{xn}
\widehat{x}(n)=\frac{1}{2\pi}\int_{-\pi}^{\pi}x(t)e^{-int}dt, \,\,\  n\in\mathbb{Z}.
\end{equation}
It follows from \eqref{ak} and \eqref{xn} by an easy calculation that $$(Vx)(t)=\int_{-\pi}^{\pi}f(t-s)x(s)ds,$$
where
\begin{equation}\label{ft}
f(t)=\frac{1}{2\pi}\sum_{n\in\mathbb{Z}}a(n)e^{int}.
\end{equation}
Thus, \eqref{T-oper on the interval} implies
$$
(T_{[-\pi,\pi]}(V)x)(t)=\int_{-\pi}^{\pi}f(t-s){\rm sgn}(t-s)x(s)ds,$$
where sgn is the sign function, i.e.
$$\text{sgn}(t):=\left\{
            \begin{array}{ll}
              1, & t>0 \hbox{,} \\
              0, & t=0 \hbox{,} \\
              -1, & t<0 \hbox{.}
            \end{array}
          \right.
$$
If
\begin{equation}\label{T(V)}
 f(t)\cdot{\rm sgn}(t)=\frac{1}{2\pi}\sum_{n\in\mathbb{Z}}b(n)e^{int},
\end{equation}
then similar to \eqref{ak}, we have
\begin{equation}\label{bk}
T_{[-\pi,\pi]}(V)=b(\mathcal{D}).
\end{equation}
Let us now identify $b.$
Fix $n\in \mathbb{Z},$ multiplying both sides of \eqref{T(V)} by the function $e^{-int},$ and integrating over the interval $[-\pi,\pi),$ we obtain
$$b(n)=\frac1{2\pi}\int_{-\pi}^{\pi}f(t)\cdot{\rm sgn}(t)e^{-int}dt.$$
Thus, by \eqref{ft}, we have
\begin{equation}\label{bn-sum}
b(n)=\frac1{2\pi}\sum_{k\in\mathbb{Z}}a(k)\cdot\int_{-\pi}^{\pi}e^{i(k-n)t}\cdot{\rm sgn}(t)dt, \,\  n\in\mathbb{Z}.
\end{equation}
Clearly,
$$\int_{-\pi}^{\pi}e^{i(k-n)t}\cdot{\rm sgn}(t)dt=2i\int_0^{\pi}\sin((k-n)t)dt.$$
If $k=n,$ then
$$\int_0^{\pi}\sin((k-n)t)dt=0.$$
If $k>n,$ then
$$\int_0^{\pi}\sin((k-n)t)dt=\frac1{k-n}\int_0^{(k-n)\pi}\sin(t)dt=\frac2{k-n}\delta_{k-n+1{\rm mod}2},$$
where $\delta_{k-n+1{\rm mod}2}=\frac{(-1)^{k-n+1}+1}{2}, \,\ k,n\in\mathbb{Z}.$
If $k<n,$ then also
$$\int_0^{\pi}\sin((k-n)t)dt=\frac2{k-n}\delta_{k-n+1{\rm mod}2}.$$
Therefore, combining above three cases, from \eqref{bn-sum}, we obtain
\begin{equation}\label{bn1}
b(n)=\frac2{\pi i}\sum_{\substack{k\in\mathbb{Z}\\ k=n+1mod2}}\frac{a(k)}{n-k}, \,\,  n\in\mathbb{Z}.
\end{equation}
Therefore, by \eqref{H dis} and \eqref{bk}, we obtain the desired result.
\end{proof}

\begin{lem}\label{second lem for T} For each $a\in\ell_{\infty}(\mathbb{Z})$  there exists a sequence $c$ such that $\mu(a)=\mu(c)$ and
$$S^{d}\mu(a))\leq c_{abs}\mu(\mathcal{H}_{d}c).$$
\end{lem}
\begin{proof}
Let $a\in\ell_{\infty}(\mathbb{Z}),$ for any $k\in\mathbb{Z}$ define
\begin{equation}\label{ck}
c(k)=\left\{
       \begin{array}{ll}
         0, & \,\,\ k>0 \hbox{,} \\
         \mu(-\frac{k}{2},a), & \,\,\ k\leq0, \, k=0mod2 \hbox{,} \\
         0, \,\,\ & \,\,\ k\leq0, \, k=1mod2 \hbox{.}
       \end{array}
     \right.
\end{equation}
and define an operator $V=c(\mathcal{D})$ on $L_{2}(-\pi,\pi).$ Hence, by \eqref{H dis} and functional calculus, we have

\begin{equation}\label{bn}
(\mathcal{H}_{d}c)(n)=\frac2{\pi i}\sum_{\substack{k\in\mathbb{Z}\\ k=n+1mod2}}\frac{c(k)}{n-k}, \,\, n\in\mathbb{Z}
\end{equation} and
$$\mu(c)=\mu(a).$$
Let $n\geq0$ such that $n=0mod2,$ then by \eqref{ck} and \eqref{bn}, we have
\begin{eqnarray*}\begin{split}|(\mathcal{H}_{d}c)(n)|
&=\frac2{\pi }\sum_{k\geq0}\frac{\mu(k,a)}{n+2k}\geq\frac2{\pi }\sum_{k\geq0}\frac{\mu(k,a)}{2(n+1)+2k}\\
&\geq\frac1{2\pi }\sum_{k\geq0}\mu(k,a)\min\{\frac{1}{n+1},\frac{1}{k}\}
=\frac1{2\pi }S^{d}\mu(n,a).
\end{split}\end{eqnarray*}
Hence,
$$|(\mathcal{H}_{d}c)(n)|\geq \left\{
               \begin{array}{ll}
                \frac1{2\pi}S^{d}\mu(n,a), & n\geq0,\, n=0mod2 \hbox{,} \\
                 0, & \text{otherwise} \hbox{.}
               \end{array}
             \right.
$$
Taking decreasing rearrangement from the last inequality and using the fact
$(S^{d}a)(n)\leq (S^{d}a)(n/2)\leq 2\cdot(S^d a)(n)$ for any positive sequence $a=\{a(n)\}_{n\in \mathbb{Z}_{+}},$ we obtain
$$ S^{d}\mu(a)\leq c_{abs} \mu(\mathcal{H}_{d}c).$$
\end{proof}

\begin{proof}[Proof of Theorem \ref{T theorem}] By \eqref{rest of T oper}, we may consider $T_{[-\pi,\pi]}$ instead of $T.$ Let $a\in\ell_{\infty}(\mathbb{Z}).$ Then, by Lemma \ref{second lem for T} there is a sequence $c$ such that $\mu(a)=\mu(c)$ and
$$ S^{d}\mu(a)\leq c_{abs} \mu(\mathcal{H}_{d}c).$$
  Define an operator $V=c(\mathcal{D})$ on $L_{2}(-\pi,\pi).$  Since $\mu(c(\mathcal{D}))=\mu(c)$ and $\mu(\mathcal{H}_{d}c)=\mu((\mathcal{H}_{d}c)(\mathcal{D})),$ it follows from the preceding inequality that
  $$ S^{d}\mu(V)=S^{d}\mu(a)\leq c_{abs} \mu(\mathcal{H}_{d}c)=c_{abs} \mu((\mathcal{H}_{d}c)(\mathcal{D})).$$
Therefore, using Lemma \ref{first lem for T}, we conclude the proof.

\end{proof}

\section{optimal symmetric quasi-banach range for the operator $S$}
In this section, we describe the optimal symmetric quasi-Banach function range for the Calder\'{o}n operator $S$ defined in \eqref{S}. We need the following lemma.
\begin{lem}\label{measure conv}
Let $\{x_n\}_{n=1}^{\infty}\subset S(0,\infty).$
If the series
$$\sum_{n=1}^{\infty}\sigma_{2^{n}}\mu(x_n)$$
converges almost everywhere (a.e.) in $S(0,\infty),$ then the series
$\sum_{n=1}^{\infty}x_{n}$ converges in measure in $S(0,\infty)$ and we have
\begin{equation}\label{est conv}
\mu\Big(\sum_{n=1}^{\infty}x_n\Big)\leq\sum_{n=1}^{\infty}\sigma_{2^{n}}\mu(x_n).
\end{equation}
\end{lem}
\begin{proof} Fix $\varepsilon,\delta>0,$ and choose $N=N(\varepsilon,\delta)$ such that
\begin{equation}\label{epsilon delta est}
\left(\sum_{n=N}^{\infty}\sigma_{2^{n}}\mu(x_n)\right)(\varepsilon)<\delta.
\end{equation}
Then, for any $N_1,N_2\geq N$ and by (2.23) in \cite[Corollary II.2, p. 67]{KPS} and \eqref{epsilon delta est}, we have
\begin{equation}\label{partial sum est}
\begin{split}\mu\left(\varepsilon,\sum_{n=N_{1}+1}^{N_2}x_{n}\right)
&=\mu\left(\varepsilon\cdot\frac{\sum_{n=N_{1}+1}^{N_2}2^{-n}}{\sum_{m=N_{1}+1}^{N_2}2^{-m}},\sum_{n=N_{1}+1}^{N_2}x_{n}\right)\\
&\stackrel{(2.23)}\leq\sum_{n=N_{1}+1}^{N_2}\mu\left(\varepsilon\cdot\frac{2^{-n}}{\sum_{m=N_{1}+1}^{N_2}2^{-m}},x_{n}\right)\\
&\leq\sum_{n=N_{1}+1}^{N_2}\mu(\varepsilon\cdot2^{N_{1}-n},x_{n})\\
&\leq\sum_{n=N}^{\infty}\mu(\varepsilon\cdot2^{-n},x_{n})=\left(\sum_{n=N}^{\infty}\sigma_{2^{n}}\mu(x_n)\right)(\varepsilon)\stackrel{\eqref{epsilon delta est}}<\delta.
\end{split}
\end{equation}
Let us denote $a_{N_1}=\sum_{n=1}^{N_1}x_{n}$ and  $a_{N_2}=\sum_{n=1}^{N_2}x_{n}.$
Then, by the preceding inequality for any $N_{1},N_{2}\geq N,$ we obtain
$$a_{N_2}-a_{N_{1}}\in U(\varepsilon,\delta):=\{x\in S(0,\infty): m(\{|x|>\delta\})<\varepsilon\},$$
which shows that $\{a_k\}_{k=1}^{\infty}$ is a Cauchy sequence in measure in $S(0,\infty).$
 Since $S(0,\infty)$ is complete in measure topology, it follows that
the series $\sum_{n=1}^{\infty}x_{n}$ converges in measure.
Therefore, since the decreasing rearrangement $\mu$ is continuous from the right, it follows from \eqref{partial sum est} that
$$\mu\left(\sum_{n=1}^{\infty}x_n\right)\leq\sum_{n=1}^{\infty}\sigma_{2^{n}}\mu(x_n).$$
\end{proof}

\begin{definition}\label{quasi-banach range} Let $E$ be a quasi-Banach symmetric space on $(0,\infty).$ Let $E(0,\infty)\subset \Lambda_{\log}(0,\infty)$ and let $S$ be the operator defined in \eqref{S} . Define
$$F(0,\infty):=\{x\in(L_{1,\infty}+L_{\infty})(0,\infty):  \exists y\in E(0,\infty), \, \mu(x)\leq S\mu(y)\}$$
such that
$$\|x\|_{F(0,\infty)}:=\inf\{\|y\|_{E(0,\infty)}:\mu(x)\leq S\mu(y)\}<\infty.$$
\end{definition}

The following result provides solution to Problem \ref{main problem} in the special case $\mathcal{M}=\mathbb{C}.$

\begin{thm}\label{quasi-banach opt range} Let $E$ be a quasi-Banach symmetric space on $(0,\infty).$ If $E(0,\infty)\subset\Lambda_{\log}(0,\infty),$  then
\begin{enumerate}[{\rm (i)}]
\item $(F(0,\infty),\|\cdot\|_{F(0,\infty)})$ is a quasi-Banach space.
\item Moreover, $(F(0,\infty),\|\cdot\|_{F(0,\infty)})$ is the optimal symmetric quasi-Banach range for the operator $S$ on $E(0,\infty).$
\end{enumerate}
\end{thm}
First we need following Lemmas.
\begin{lem}\label{linearity of op range} Let $E$ be a symmetric space on $(0,\infty).$ If $E(0,\infty)\subset\Lambda_{\log}(0,\infty),$ then $F(0,\infty)$ is a linear space.
\end{lem}
\begin{proof} For $j=1,2,$ let $x_{j}\in F(0,\infty),$ and $\alpha_{j}$ be any scalars from the field of complex numbers. Then by the definition of $F(0,\infty),$ there exist corresponding $y_{j}\in E(0,\infty)$ such that
 $\mu(x_{j})\leq S\mu(y_{j}).$ Therefore, for any $x_{j}\in F(0,\infty)$ and $\alpha_{j},$  by \cite[Proposition II.1.7, p. 41]{BSh}, we have
\begin{equation}\begin{split}\label{linearity}\mu(\alpha_{1}x_{1}+\alpha_{2}x_{2})
&\leq\sigma_{2}\mu(\alpha_{1}x_{1})+\sigma_{2}\mu(\alpha_{2}x_2)=|\alpha_{1}|\cdot\sigma_{2}\mu(x_{1})+|\alpha_{2}|\cdot\sigma_{2}\mu(x_{2})\\
&\leq |\alpha_{1}|\cdot\sigma_{2}\left(S(\mu(y_{1}))\right)+|\alpha_{2}|\cdot\sigma_{2}\left(S(\mu(y_{2}))\right)\\
&= S(|\alpha_{1}|\cdot\sigma_{2}\mu(y_{1}))+S(|\alpha_{2}|\cdot\sigma_{2}\mu(y_{2}))\\
&=S(|\alpha_{1}|\cdot\sigma_{2}\mu(y_{1})+|\alpha_{2}|\cdot\sigma_{2}\mu(y_{2})).
\end{split}\end{equation}
Since $E(0,\infty)$ is a linear space and  $|\alpha_{1}|\cdot\sigma_{2}\mu(y_{1})+|\alpha_{2}|\cdot\sigma_{2}\mu(y_{2})\in E(0,\infty)$, it follows that $\alpha_{1}x_{1}+\alpha_{2}x_{2}\in F(0,\infty).$ This shows that
$F(0,\infty)$ is a linear space.
\end{proof}

\begin{lem}\label{quasi-norm of op range} Let $E$ be a quasi-Banach symmetric space on $(0,\infty).$ If $E(0,\infty)\subset\Lambda_{\log}(0,\infty),$ then $(F(0,\infty),\|\cdot\|_{F(0,\infty)})$ is a
quasi-normed space.
\end{lem}
\begin{proof}Let us prove that the expression
\begin{equation}\label{norm exp}\|x\|_{F(0,\infty)}:=\inf\{\|y\|_{E(0,\infty)}:\mu(x)\leq S\mu(y)\}
\end{equation}
defines a quasi-norm in $F(0,\infty).$ Clearly, if $x=0,$ then, by \eqref{norm exp} we obtain that $\|x\|_{F(0,\infty)}=0,$ and for every scalar $\alpha,$ we have $\|\alpha
x\|_{F(0,\infty)}=|\alpha|\cdot\|x\|_{F(0,\infty)}.$ We shall prove the non-trivial part. If $\|x\|_{F(0,\infty)}=0,$ then there exists $y_n$ in $E(0,\infty)$ with $\mu(x)\leq S\mu(y_n)$ such that
$\|y_n\|_{E(0,\infty)}\rightarrow0,$ as $n\rightarrow \infty.$ By assumption, we have that $S:E(0,\infty)\rightarrow (L_{1,\infty}+L_{\infty})(0,\infty)$ (see \ref{S}). Since $S$ is positive operator (see \cite[Chapter
III, p. 134]{BSh}), it follows from \cite[Proposition 1.3.5, p. 27]{MN} that
 $S:E(0,\infty)\rightarrow (L_{1,\infty}+L_{\infty})(0,\infty)$ is bounded. Hence, $\|S\mu(y_n)\|_{(L_{1,\infty}+L_{\infty})(0,\infty)}\rightarrow 0$ as $n\rightarrow \infty.$  From the condition $\mu(x)\leq S\mu(y_n),$ for
 every $n\in\mathbb{Z}_{+},$ we have
$$\|\mu(x)\|_{(L_{1,\infty}+L_{\infty})(0,\infty)}\leq\|S\mu(y_n)\|_{(L_{1,\infty}+L_{\infty})(0,\infty)}\rightarrow 0,$$
which shows that $x=0.$

Let us prove that $\|\cdot\|_{F(0,\infty)}$ satisfies quasi-triangle inequality. For $j=1,2,$ let $x_{j}\in F(0,\infty),$ and fix $\varepsilon>0.$ Then there exist $y_{j}\in E(0,\infty)$ with
$\mu(x_{j})\leq S\mu(y_{j})$ and $\|y_{j}\|_{E(0,\infty)}<\|x_{j}\|_{F(0,\infty)}+\varepsilon.$
Hence, from \eqref{norm exp} and since $\|\cdot\|_{E(0,\infty)}$ is quasi-norm, it follows from  \eqref{linearity} and \cite[Remark 18]{F}
 that
\begin{eqnarray*}\begin{split}\|x_{1}+x_{2}\|_{F(0,\infty)}
&\leq\|\sigma_{2}\mu(y_{1})+\sigma_{2}\mu(y_{2})\|_{E(0,\infty)}=\|\sigma_{2}\left(\mu(y_{1})+\mu(y_{2})\right)\|_{E(0,\infty)}\\
&\leq2\cdot c_{E}\|\mu(y_{1})+\mu(y_{2})\|_{E(0,\infty)}\leq2\cdot c^{2}_{E}\left(\|y_{1}\|_{E(0,\infty)}+\|y_{2}\|_{E(0,\infty)}\right)
\end{split}
\end{eqnarray*}
and by choice of $y_{1},y_{2}\in E(0,\infty)$
$$\|x_{1}+x_{2}\|_{F(0,\infty)}\leq2\cdot c^{2}_{E}\left(\|x_{1}\|_{F(0,\infty)}+\|x_{2}\|_{F(0,\infty)}\right)+4c^{2}_{E}\cdot\varepsilon.$$
Since $\varepsilon$ is arbitrary, letting $\varepsilon\rightarrow 0,$ we obtain that $\|\cdot\|_{F(0,\infty)}$ defines a quasi-norm. Thus, $(F(0,\infty),\|\cdot\|_{F(0,\infty)})$ is a linear quasi-normed
space.
\end{proof}
Now we are ready to prove Theorem \ref{quasi-banach opt range}.
\begin{proof}[Proof of Theorem \ref{quasi-banach opt range}]
 By Lemma \ref{linearity of op range} and \ref{quasi-norm of op range}, $(F(0,\infty),\|\cdot\|_{F(0,\infty)})$ is a linear quasi-normed space.

First, we prove that $F(0,\infty)$ is a quasi-Banach space. To show that $F(0,\infty)$ is quasi-Banach, it remains to see that it is complete. Since the dilation operator is bounded in any quasi-Banach
symmetric space, it follows that there is a constant $c_{E}$ depending on $E$ such that
\begin{equation}\label{dilation bound}
\|\sigma_{2^{n}}y\|_{E(0,\infty)}\leq c_{E}^{n}\|y\|_{E(0,\infty)}
\end{equation}
(see \cite[Remark 18]{F}) for all $y\in E(0,\infty)$ and $n\in \mathbb{N}.$
On the other hand, since $E(0,\infty)$ is quasi-Banach symmetric space, it follows from Aoki-Rolewicz theorem that every quasi-normed space (such as $E(0,\infty))$ is metrizable (see \cite[Theorem 1.3]{KPR}) and there exists $0<p<1,$ such that
$$\|y_{1}+y_{2}\|_{E(0,\infty)}^{p}\leq\|y_{1}\|_{E(0,\infty)}^{p}+\|y_{2}\|_{E(0,\infty)}^{p}$$
for all $y_{1},y_{2}\in E(0,\infty).$ We have to show that an arbitrary Cauchy sequence in $F(0,\infty)$
converges to an element from $F(0,\infty).$ Fix such a sequence $\{x_n\}_{n=1}^{\infty}\subset F(0,\infty).$  Without loss of generality, assume that for $\varepsilon< c_{E}^{-1},$ we have
$$\|x_{n+1}-x_{n}\|_{F(0,\infty)}\leq \varepsilon^{n} $$ and
$$\mu(x_{n+1}-x_{n})\leq S\mu(y_n)$$ such that
$\|y_n\|_{E(0,\infty)}\leq 2\cdot \varepsilon^{n}.$ Let us show that the series $\sum_{n=1}^{\infty}\sigma_{2^{n}}\mu(x_{n+1}-x_{n})$ converges a.e.
 Since $S$ is linear and commutes with the dilation operator,
it follows that
\begin{equation}\begin{split}\label{partial S est}
\sum_{n=1}^{\infty}\sigma_{2^{n}}\mu(x_{n+1}-x_{n})\leq\sum_{n=1}^{\infty}\sigma_{2^{n}}S\mu(y_{n})=\sum_{n=1}^{\infty}S(\sigma_{2^{n}}\mu(y_{n}))
=S\left(\sum_{n=1}^{\infty}\sigma_{2^{n}}\mu(y_{n})\right).
\end{split}\end{equation}
Hence,
\begin{eqnarray*}\begin{split}
\left\|\sum_{n=1}^{\infty}\sigma_{2^{n}}\mu(y_{n})\right\|_{E(0,\infty)}^{p}
\leq\sum_{n=1}^{\infty}\|\sigma_{2^{n}}\mu(y_{n})\|_{E(0,\infty)}^{p}\\
\stackrel{\eqref{dilation bound}}\leq\sum_{n=1}^{\infty}c_{E}^{np}\|y_{n}\|_{E(0,\infty)}^{p}
\leq\sum_{n=1}^{\infty}(c_{E}\cdot\varepsilon)^{np}<\infty.
\end{split}\end{eqnarray*}
Therefore, the series $\sum_{n=1}^{\infty}\sigma_{2^{n}}\mu(y_{n})$ converges in a.e. in $E(0,\infty).$ Since $S$ is continuous on $E(0,\infty)$ by assumption, it follows from \eqref{partial S est} that the series
$\sum_{n=1}^{\infty}\sigma_{2^{n}}\mu(x_{n+1}-x_{n})$ belongs to $F(0,\infty).$ Then, by Lemma \ref{measure conv}, the series $\sum_{n=1}^{\infty}(x_{n+1}-x_{n})$ converges in measure and belongs to $F(0,\infty),$
and we have
$$\|x-x_{1}\|_{F(0,\infty)}=\left\|\sum_{n=1}^{\infty}(x_{n+1}-x_{n})\right\|_{F(0,\infty)}\stackrel{\eqref{est conv}}\leq\left\|\sum_{n=1}^{\infty}\sigma_{2^{n}}\mu(x_{n+1}-x_{n})\right\|_{F(0,\infty)}<\infty.$$
 This shows that $x\in F(0,\infty).$ So, $F(0,\infty)$ is complete. On the other hand, since $\|x\|_{F(0,\infty)}=\|\mu(x)\|_{F(0,\infty)},$ it follows that $F(0,\infty)$ is a symmetric space. So, the  space $(F(0,\infty),\|\cdot\|_{F(0,\infty)})$ is a quasi-Banach symmetric space.

 Next, we prove second part of the theorem. Let $G(0,\infty)\subset S(0,\infty)$ be a symmetric module over the algebra $L_{\infty}(0,\infty)$ (i.e. if $x_1\sim x_2$ and $x_1\in G(0,\infty),$  then $x_2\in
 G(0,\infty)$). If $S:E(0,\infty)\to G(0,\infty),$ then $S\mu(x)\in G(0,\infty).$ By definition of $F(0,\infty),$ we have that $F(0,\infty)\subset G(0,\infty).$
Hence, $F(0,\infty)$ is minimal receptacle in the category of all symmetric function modules. However, $F(0,\infty)$ is a symmetric quasi-Banach space. Thus, it is the minimal receptacle in the category of
quasi-Banach spaces.
\end{proof}

\section{the non-commutative optimal symmetric quasi-banach space for the triangular truncation operator $T$}

In this section, we resolve Problems \ref{main problem} and \ref{second problem} in full generality.
 Let $E$ be a symmetric quasi-Banach sequence space on $\mathbb{Z}_{+}.$ If $E(\mathbb{Z}_{+})\subset\Lambda_{\log}(\mathbb{Z}_{+}),$ then the operator
 $$S^{d}:\Lambda_{\log}(\mathbb{Z}_{+})\rightarrow (\ell_{1,\infty}+\ell_{\infty})(\mathbb{Z}_{+})$$
  is well defined (see \eqref{S dis}).

 Let $E(\mathbb{Z}_{+})$ be a symmetric quasi-Banach sequence space and $\mathcal{E}(H)$ be the corresponding non-commutative symmetric quasi-Banach space (see \cite{F}).
If $S:E(\mathbb{Z}_{+})\rightarrow (\ell_{1,\infty}+\ell_{\infty})(\mathbb{Z}_{+}),$ then by Theorem \ref{first main theorem}, the operator $T$ (see \eqref{T-oper}) is well defined on the corresponding non-commutative symmetric quasi-Banach space
$\mathcal{E}(H).$
The optimal quasi-Banach symmetric space for the discrete Calder\'{o}n operator $S^{d}$ is constructed similarly to Definition \ref{quasi-banach range}.
\begin{definition}\label{discrete quasi-banach range} Let $E$ be a quasi-Banach symmetric sequence space on $\mathbb{Z}_{+}.$ Let $E(\mathbb{Z}_{+})\subset \Lambda_{\log}(\mathbb{Z}_{+})$ and let
$S^{d}$ be the operator defined in \eqref{S dis}. Define
$$F(\mathbb{Z}_{+}):=\{a\in(\ell_{1,\infty}+\ell_{\infty})(\mathbb{Z}_{+}):  \exists b\in E(\mathbb{Z}_{+}), \, \mu(a)\leq S^{d}\mu(b)\}$$
such that
$$\|a\|_{F(\mathbb{Z}_{+})}:=\inf\{\|b\|_{E(\mathbb{Z}_{+})}:\mu(a)\leq S^{d}\mu(b)\}<\infty.$$
\end{definition}

Let $E(\mathbb{Z}_{+})$ be a symmetric quasi-Banach sequence space.
Then the space $$E(0,\infty)=\Big\{x\in L_{\infty}(0,\infty):\left\{\int_{n}^{n+1}\mu(s,x)ds\right\}_{n\geq0}\in E(\mathbb{Z}_{+})\Big\}$$ is a symmetric quasi-Banach space
equipped with the norm
$$\|x\|_{E(0,\infty)}:=\max\Big\{\|x\|_{L_{\infty}(0,\infty)},\Big\|\left\{\int_{n}^{n+1}\mu(s,x)ds\right\}_{n\geq0}\Big\|_{E(\mathbb{Z}_{+})}\Big\}.$$
By Theorem \ref{quasi-banach opt range}, $F(0,\infty)$ is also symmetric quasi-Banach space, and we have
$$F(\mathbb{Z}_{+})=F(0,\infty)\cap\ell_{\infty}(\mathbb{Z}_{+}).$$ Hence, $F(\mathbb{Z}_{+})$ is a symmetric quasi-Banach space.


\begin{rmk}\label{coincidence opt range}
Since $F(\mathbb{Z}_{+})$ is a symmetric quasi-Banach sequence space, it follows from \cite[Theorem 4]{F} that there exists corresponding non-commutative symmetric quasi-Banach space defined
by
\begin{equation}\label{Calkin space}\mathcal{F}(H):=\{A\in B(H):\mu(A)\in F(\mathbb{Z}_{+})\}
\end{equation}
with the norm
$$\|A\|_{\mathcal{F}(H)}:=\|\mu(A)\|_{F(\mathbb{Z}_{+})}.$$
\end{rmk}

\begin{proposition} Let $\mathcal{E}(H)=\mathcal{L}_{1,\infty}(H),$ then
$$\mathcal{F}(H)=\left\{A\in (\mathcal{L}_{1,\infty}+\mathcal{L}_{\infty})(H):\exists c_{A}, \mu(n,A)\leq c_{A}\frac{\log(n+2)}{n+1}, \, n\in \mathbb{Z}_{+}\right\}.$$
\end{proposition}
\begin{proof}If $\mu(k,A)=\frac{1}{k+1}, \,\ k\in\mathbb{Z}_{+},$ then by \eqref{S dis}, we have
$$(S^{d}\mu(A))(n):=\frac{1}{n+1}\sum_{k=0}^{n}\frac{1}{k+1}+\sum_{k=n+1}^{\infty}\frac{1}{k(k+1)}$$
and
$$
(S^{d}\mu(A))(n)\approx \frac{\log(n+1)}{n+1}
$$
for large $n.$
Therefore, if $\mathcal{E}(H)=\mathcal{L}_{1,\infty}(H),$ then the optimal range for the operator $T$ is
$$\mathcal{F}(H)=\{A\in(\mathcal{L}_{1,\infty}+\mathcal{L}_{\infty})(H):\exists c_{A}, \mu(n,A)\leq c_{A}\frac{\log(n+2)}{n+1}, \, n\in \mathbb{Z}_{+}\},$$
where $c_{A}$ is a constant depending only $A.$
\end{proof}

\begin{rmk} The optimal symmetric quasi-Banach space $F(\mathbb{Z}_{+})$ (resp. $F(0,\infty)$) in Definition \ref{discrete quasi-banach range} (resp. Definition \ref{quasi-banach range}) is defined similarly on $\mathbb{Z}$ (resp. $\mathbb{R}$) and becomes symmetric quasi-Banach space.
\end{rmk}

The following theorem is a main result of this section which completely resolves Problem \ref{main problem}.
\begin{thm}\label{opt. range th H} Let $E=E(\mathbb{R})\subset\Lambda_{\log}(\mathbb{R})$ be a symmetric quasi-Banach function space and let $\mathcal{M}$ be a semifinite atomless von Neumann algebra. If
 $F$ is given by Definition \ref{quasi-banach range}, then
 the space $\mathcal{F}(\mathcal{M}\bar{\otimes}L_{\infty}(\mathbb{R}))$ is the optimal symmetric quasi-Banach range for the Hilbert transform $1\otimes\mathcal{H}$ on $\mathcal{E}(\mathcal{M}\bar{\otimes}L_{\infty}(\mathbb{R})).$
\end{thm}
\begin{proof}
By Theorem 4.3 (or Corollary 4.6) in \cite{Ran2}, the operator $1\otimes\mathcal{H}$ satisfies the conditions of Theorem \ref{first main theorem}. Hence, by Theorem \ref{first main theorem}, we have
$$\mu((1\otimes\mathcal{H})(x))\leq c_{{\rm abs}}S\mu(x)$$
for all $x\in\Lambda_{\log}(\mathcal{M}\bar{\otimes} L_{\infty}(\mathbb{R})).$ By the definition of $F$ (see Definition \ref{quasi-banach range}), we have that $(1\otimes\mathcal{H})(x)\in \mathcal{F}(\mathcal{M}\bar{\otimes}L_{\infty}(\mathbb{R}))$ for all $x\in\Lambda_{\log}(\mathcal{M}\bar{\otimes} L_{\infty}(\mathbb{R})).$

Now, suppose that $G(\mathbb{R})$ is another symmetric quasi-Banach space such that $1\otimes\mathcal{H}:\mathcal{E}(\mathcal{M}\bar{\otimes}L_{\infty}(\mathbb{R}))\rightarrow \mathcal{G}(\mathcal{M}\bar{\otimes}L_{\infty}(\mathbb{R}))$ is bounded. It follows immediately that $\mathcal{H}:E(\mathbb{R})\to G(\mathbb{R}).$ Take $x\in E(\mathbb{R}).$ By \cite[Proposition III. 4.10, p. 140]{BSh} there exists $y$ with $\mu(x)=\mu(y)$ such that $S\mu(x)\leq c_{abs}\mu(\mathcal{H}y),$ which shows that $S\mu(x)\in G(0,\infty).$ Since $x\in E(\mathbb{R})$ is arbitrary, it follows from \eqref{S proper} that
$S:E(0,\infty)\rightarrow F(0,\infty).$ Hence, $F(0,\infty)\subset G(0,\infty).$
\end{proof}

If $\mathcal{M}=\mathbb{C},$ then the result of Theorem \ref{opt. range th H} coincides with that of Theorem \ref{quasi-banach opt range}.

The following theorem gives a solution to Problem \ref{second problem}.
\begin{thm}\label{opt. range th T} Let $E=E(\mathbb{Z}_{+})\subset\Lambda_{\log}(\mathbb{Z}_{+})$ be a symmetric quasi-Banach sequence space and $\mathcal{E}(H)$ be the corresponding non-commutative symmetric quasi-Banach ideal. If $\mathcal{F}(H)$ is given by Remark \ref{coincidence opt range}, then the space $\mathcal{F}(H)$ is the optimal symmetric quasi-Banach range for the operator $T$ on $\mathcal{E}(H).$

\end{thm}
\begin{proof}
First, let us see that $T:\mathcal{E}(H)\rightarrow \mathcal{F}(H)$ is bounded. If $A\in \mathcal{E}(H),$  then by Theorem \ref{weak est for T}, $T$ satisfies the assumptions of Theorem \ref{first main theorem}. Therefore, we have $\mu(T(A))\leq c_{abs} S^{d}\mu(A).$ By Remark \ref{coincidence opt range}, we have
\begin{eqnarray*}\begin{split}
\|T(A)\|_{\mathcal{F}(H)}
&\leq c_{abs}\|S^{d}\mu(A)\|_{F(\mathbb{Z}_{+})}=c_{abs}\inf\{\|b\|_{E(\mathbb{Z}_{+})}:S^{d}\mu(A)\leq S^{d}\mu(b)\}\\
&=c_{abs}\inf\{\|B\|_{\mathcal{E}(H)}:S^{d}\mu(A)\leq S^{d}\mu(B)\}\leq c_{abs}\|A\|_{\mathcal{E}(H)}.
\end{split}\end{eqnarray*}
Hence, $T:\mathcal{E}(H)\rightarrow \mathcal{F}(H)$ is bounded.
Now, suppose that $\mathcal{G}(H)$ is another symmetric quasi-Banach ideal such that $T:\mathcal{E}(H)\rightarrow \mathcal{G}(H)$ is bounded, and
let us show that $\mathcal{F}(H)\subset \mathcal{G}(H).$
Let $a\in E(\mathbb{Z}_{+}).$ By Theorem \ref{T theorem}, there exists an operator $A$ such that $\mu(a)=\mu(A)$ and $S^{d}\mu(a)\leq c_{abs}\mu(T(A)).$ Since $T(A)\in \mathcal{G}(H),$ it follows that $S^{d}\mu(a)\in
G(\mathbb{Z}_{+}).$

Therefore, we have that $\mathcal{F}(H)\subset \mathcal{G}(H)$ as claimed.
\end{proof}

Let us denote
$$(L_{1,\infty}(0,\infty))^{0}:=\{x\in L_{1,\infty}(0,\infty):\lim_{t\rightarrow0+} t\mu(t,x)=0\}.$$

The following proposition shows that the optimal range for the Hilbert transform on $L_{1}(\mathbb{R})$ is $(L_{1,\infty}(\mathbb{R}))^{0}.$
In particular, this result refines the classical Kolmogorov's theorem \cite[Theorem III.4.9 (b), p. 139]{BSh}.

\begin{proposition}\label{optim weak l1} If $E(0,\infty)=L_{1}(0,\infty),$ then
$$F(0,\infty)=(L_{1,\infty}(0,\infty))^{0}.$$
\end{proposition}
\begin{proof}

 If $x\in F(0,\infty),$ then by Definition \ref{quasi-banach range}, there exist $y\in L_{1}(0,\infty)$ such that
$$t\mu(t,x)\leq\int_{0}^{t}\mu(s,y)ds+t\int_{t}^{\infty}\mu(s,y)\frac{ds}{s}, \, t>0.$$
By Dominated Convergence Theorem, we have
$t\mu(t,x)\rightarrow0$ as $t\rightarrow0+.$ Thus,
\begin{equation}\label{subset}
F(0,\infty)\subset(L_{1,\infty}(0,\infty))^{0}.
\end{equation}
To see the converse inclusion, take $x\in L_{1,\infty}(0,\infty),$ $\|x\|_{L_{1,\infty}(0,\infty)}=1,$ such that $\lim_{t\rightarrow0+}t\mu(t,x)=0.$ We have to find $y\in L_{1}(0,\infty)$ such that $\mu(x)\leq S\mu(y).$ For
$0<t<1,$ define
 $$f(t)=\sup_{0<s<t}(s\mu(s,x)).$$
 It is clear that $f$ is increasing, positive function, and $t\mu(t,x)\leq f(t),$ $0<t<1.$ By the definition of $f,$ and the hypothesis on $x,$ it can be seen that $f(0+)=0.$
 Let
 $$h(t)=f(2^{n+1})\cdot\frac{2^{n+1}-t}{2^{n+1}-2^{n}}+f(2^{n+2})\cdot\frac{t- 2^{n}}{2^{n+1}-2^{n}}, \,\, t\in(2^{n},2^{n+1}), \, n\in \mathbb{Z}$$
Then, it is easy to see that
 $h(t)\leq f(4t)$ and $f(t)\leq h(t).$
   Define (a.e.)
$$y(t):=\left\{
         \begin{array}{ll}
           h'(t), & 0<t<1, \\
           0, & 1\leq t<\infty.
         \end{array}
       \right.
$$ Clearly, $y\in L_{1}(0,\infty).$
Moreover, if $0<t<1,$ then
$$\mu(t,x)\leq \frac{1}{t}h(t)=\frac{1}{t}\int_{0}^{t}y(s)ds\leq\frac{1}{t}\int_{0}^{t}\mu(s,y)ds\leq S\mu(y)(t),$$
and for $1\leq t<\infty,$
$$\mu(t,x)\leq \frac{1}{t}=\frac{1}{t\cdot h(1)}\int_{0}^{t}\mu(s,y)ds\leq \frac{S\mu(y)(t)}{h(1)}.$$
Therefore, $\mu(t,x)\leq S\mu(y)(t), \, t>0.$  This shows that
$$F(0,\infty)\supset(L_{1,\infty}(0,\infty))^{0}.$$
Combining \eqref{subset} and preceding inclusion, we obtain
$$F(0,\infty)=(L_{1,\infty}(0,\infty))^{0}.$$
\end{proof}

The following proposition describes the optimal range for the triangular truncation operator $T$ on $\mathcal{L}_{1}(H).$
\begin{proposition}\label{nc optim weak l1} If $\mathcal{E}(H)=\mathcal{L}_{1}(H),$ then the optimal range for the triangular truncation operator $T$ is
$$\mathcal{F}(H)=\mathcal{L}_{1,\infty}(H).$$
\end{proposition}
\begin{proof} Let $\mathcal{E}=\mathcal{L}_{1}(H).$ If $A_{0}\in \mathcal{L}_{1}(H),$ then $\mu(A_{0})\in\ell_{1}(\mathbb{Z}_{+}).$ Hence, $S^{d}a_{0}\in F(\mathbb{Z}_{+}).$  Since $S^{d}\mu(A_{0})$ is equivalent to the sequence $\Big\{\frac{1}{n+1}\Big\}_{n\geq0},$ it follows that $\ell_{1,\infty}(\mathbb{Z}_{+})= F(\mathbb{Z}_{+}),$ which is by Remark \ref{coincidence opt range} that $\mathcal{L}_{1,\infty}(H)=\mathcal{F}(H).$
\end{proof}

Preceding result is similar to the classical Macaev's theorem \cite[Theorem VII.5.1, p. 345]{GK2}.

\section{Applications.}

In this section, we show important applications of our approach in previous sections to Double Operator Integrals (see Subsection \ref{doi subsection}) associated with Lipschitz functions $f$ defined on $\mathbb{R}.$
The following theorem complements \cite[Theorem 1.2]{CPSZ1},\cite[Theorem 1.2]{CPSZ},\cite[Theorem 8]{Da1},\cite[Theorem 2.2 (i),(ii), Lemma 2.3, and Theorem 3.4 (ii),(iii)]{DDdePS},\cite[Theorem 3.3 and Corollary 3.4]{DDPS},\cite[Theorem 6 (ii), Corollary 7, see also Theorems 12 and 13]{Ko2},\cite[Theorem 2.5 (i)]{NP} and \cite[Theorem 1]{PS}. For simplicity, we state the result in the case, when $\mathcal {M}=B(H).$
\begin{thm} Let $\mathcal{E}(H)$ and $\mathcal{F}(H)$ be as in Theorem \ref{opt. range th T}.
The following assertions hold
\begin{enumerate}[{\rm (i)}]
         \item If $A=A^{*}$ is a self-adjoint operator in $B(H),$ then the double operator integral (associated with a Lipschitz function $f$ defined on $\mathbb{R}$) $T_{f^{[1]}}^{A,A}:\mathcal{E}(H)\rightarrow \mathcal{F}(H)$ is bounded and
$$\|T_{f^{[1]}}^{A,A}\|_{\mathcal{E}(H)\rightarrow \mathcal{F}(H)}<c_{E}\|f'\|_{L_{\infty}(\mathbb{R})};$$
\item

For all self-adjoint operators $A,B\in B(H)$ such that $[A,B]\in \mathcal{E}(H)$ and for every Lipschitz function $f$ defined on $\mathbb{R},$ we have
$$\|[f(A),B]\|_{\mathcal{F}(H)}\leq c_{abs}\|f'\|_{L_{\infty}(\mathbb{R})}\|[A,B]\|_{\mathcal{E}(H)},$$
where $[A,B]:=AB-BA.$
For all self-adjoint operators $X,Y\in B(H)$ such that $X-Y\in \mathcal{E}(H)$ and for every Lipschitz function $f$ defined on $\mathbb{R},$ we have
$$\|f(X)-f(Y)\|_{\mathcal{F}(H)}\leq c_{abs}\|f'\|_{L_{\infty}(\mathbb{R})}\|X-Y\|_{\mathcal{E}(H)}.$$

\end{enumerate}
\end{thm}

\begin{proof}
 By Theorem \ref{T theorem}, the operator $S^d$ acts boundedly from $E(\mathbb{Z}_{+})$ into $F(\mathbb{Z}_{+}).$ Since the Double operator integral $T_{f^{[1]}}^{A,A}$ (see Subsection \ref{doi subsection}) associated with a self-adjoint operator
$A\in B(H)$ satisfies the assumptions of the Theorem \ref{first main theorem} (ii) (see \cite[Theorem 1.2]{CPSZ}), it follows that
\begin{eqnarray*}\begin{split}
\|T_{f^{[1]}}^{A,A}(V)\|_{\mathcal{F}(H)}&\leq c_{abs}\|f'\|_{L_{\infty}(\mathbb{R})}\|S^d\mu(V)\|_{F(\mathbb{Z}_{+})}\\
&\leq c_{abs}\|f'\|_{L_{\infty}(\mathbb{R})}\|\mu(V)\|_{E(\mathbb{Z}_{+})}
=c_{abs}\|f'\|_{L_{\infty}(\mathbb{R})}\|V\|_{\mathcal{E}(H)}, \,\ V\in \mathcal{E}(H).
\end{split}\end{eqnarray*}
In other words, $T_{f^{[1]}}^{A,A}:\mathcal{E}(H)\rightarrow \mathcal{F}(H)$ is bounded.

Let us prove $(ii).$ In fact, $(ii)$ follows from $(i).$
The commutator estimate follows from the observation that the double operator integral $T^{A,A}_{f^{[1]}}([A,B])$ is equal to $[f(A),B]$  for the operators $A,B\in B(H)$ such that $[A,B]\in \mathcal{E}(H)$ (see
\cite[Lemma 5.2]{CPSZ}). Therefore, $(ii)$ follows from $(i).$ Finally, as explained in the proof of \cite[Theorem 5.3]{CPSZ}, Lipschitz estimates follow from commutator estimates.
\end{proof}

Let $\mathcal{T}$ be the abstract operator defined in Convention \ref{main setting}. Let $C$ be the Ces\`{a}ro operator defined in \eqref{Ces}.
For brevity, we will denote the notion $\max\{\log(x),0\}$ by $\log_{+}(x).$
The following theorem, which describes a non-commutative analogue of the classical Zygmund's theorem \cite[Theorem V.6.6 (a), p. 248]{BSh} (see also \cite[Corollary IV.6.9, p. 251]{BSh}), was earlier proved in \cite[Theorem 2.5]{HU}. However, for convenience of the reader, we present below a complete proof based on a different approach from that of \cite{HU}.
\begin{thm} \label{L1 est} Let $\mathcal{M}$ be a von Neumann algebra equipped with a faithful normal finite trace $\tau,$ i.e. $\tau(1)<\infty.$ If $A\in \mathcal{L}_{1}(\mathcal{M})$ is a positive operator such that $\|\mu(A)\log_{+}(\mu(A))\|_{L_{1}(0,1)}<\infty,$
then we have
\begin{eqnarray*}\begin{split}\|\mathcal{T}(A)\|_{\mathcal{L}_{1}(\mathcal{M})}
\leq c_{abs}\left(1+\|\mu(A)\log_{+}(\mu(A))\|_{L_{1}(0,1)}\right).
\end{split}\end{eqnarray*}
\end{thm}
\begin{proof} The proof will be divided into several steps.

{\bf Step 1.}
Let $L_{M}(0,1)$ be Orlicz space (see \cite[Chapter IV.8, pp. 265-279]{BSh}) associated with the Young's function $M(t)=t\log(1+t)$ (see \cite[Definition IV.8.5, p. 265]{BSh}) and let $\Lambda_{\varphi}(0,1)$ be the
Lorentz space associated with the function $\varphi(t):=t\log(\frac{e}{t}),$ $t>0.$
First, we claim that
$$\Lambda_{\varphi}(0,1)=L_{M}(0,1)$$

with equivalent norms.

It is easy to see that Young's functions (see \cite[Definition IV. 8.1, p. 271]{BSh})
$M:t\rightarrow t\log(1+t)$ and $N:t\rightarrow e^{t}-1$ are complementary.
Hence, by \cite[Corollary IV.8.15, p. 275]{BSh}, we have
$$L_{M}(0,1)=L_{N}(0,1)^{\times}.$$
On the other hand, by \cite[Lemma 4.3]{AS}, we obtain
$$L_{N}(0,1)=M_{\varphi}(0,1)=\Lambda_{\varphi}(0,1)^{\times},$$
where
$$M_{\varphi}(0,1):=\{x\in S(0,1):\sup_{t>0}\frac{1}{\varphi(t)}\int_{0}^{1}\mu(s,x)ds<\infty\}$$
is the Marcinkiewicz space associated with the function $\varphi(t):=t\log(\frac{e}{t}),$ $t>0$ (see \cite[Chapter II.5, pp. 112-118]{KPS} for more details).
Thus,
\begin{equation}\label{Orlich equivalence} L_{M}(0,1)=\Lambda_{\varphi}(0,1)^{\times\times}=\Lambda_{\varphi}(0,1),
\end{equation}
(see \cite[Chapter II.5, p. 114]{KPS}).

{\bf Step 2.}

If $A\geq0$ is such that $M(A)\in \mathcal{L}_{1}(\mathcal{M}),$ then
\begin{equation}\label{est of Ces}\|A\|_{\mathcal{L}_{M}(\mathcal{M})}
\leq c_{abs}\left(1+\|M(A)\|_{\mathcal{L}_{1}(\mathcal{M})}\right),
\end{equation}
where $\|A\|_{\mathcal{L}_{M}(\mathcal{M})}:=\|\mu(A)\|_{L_{M}(0,1)}.$
Indeed, if $\|A\|_{\mathcal{L}_{M}(\mathcal{M})}\leq 1,$ then there is nothing to prove. If $ \|A\|_{\mathcal{L}_{M}(\mathcal{M})}\geq 1,$ then there exists $\lambda$ such that $\lambda=\|A\|_{\mathcal{L}_{M}(\mathcal{M})}\geq1$ and we have
$\|M(\frac{A}{\lambda})\|_{\mathcal{L}_{1}(\mathcal{M})}=1.$ Since $M(\frac{A}{\lambda})\leq\frac{1}{\lambda}M(A),$ it follows that
$$\|M(A)\|_{\mathcal{L}_{1}(\mathcal{M})}\geq\lambda=\|A\|_{\mathcal{L}_{M}(\mathcal{M})}.$$

{\bf Step 3.}

By Theorem \ref{first main theorem}, we have $\mu(\mathcal{T}(A))\leq c_{abs}S\mu(A).$
Therefore, it is sufficient to estimate $S\mu(A).$ By definition of $S$ (see \eqref{S}), we have
$$S\mu(A)=C\mu(A)+C'\mu(A).$$
By an easy calculation, we obtain
$$\|C'\mu(A)\|_{L_{1}(0,1)}=\|\mu(A)\|_{L_{1}(0,1)}=\|A\|_{\mathcal{L}_{1}(\mathcal{M})}, \, A\in \mathcal{L}_{1}(\mathcal{M})$$
and
$$\|C\mu(A)\|_{L_{1}(0,1)}=\|\mu(A)\|_{\Lambda_{\varphi}(0,1)}, \,\ A\in \Lambda_{\varphi}(\mathcal{M})$$
(see also (6.7) in \cite[Chapter IV.6, p. 245]{BSh}).
Therefore,
it follows from \eqref{est of Ces} and \cite[Theorem IV.6.5, p. 247]{BSh} that
\begin{eqnarray*}\begin{split}
\|\mathcal{T}(A)\|_{\mathcal{L}_{1}(\mathcal{M})}
&\leq c_{abs}\|S\mu(A)\|_{L_{1}(0,1)}\leq c_{abs}\left(\|C\mu(A)\|_{L_{1}(0,1)}+\|\mu(A)\|_{L_{1}(0,1)}\right)\\
&=c_{abs}\left(\|A\|_{\Lambda_{\varphi}(\mathcal{M})}+\|A\|_{\mathcal{L}_{1}(\mathcal{M})}\right)
\approx c_{abs}\left(\|A\|_{\mathcal{L}_{M}(\mathcal{M})}+\|A\|_{\mathcal{L}_{1}(\mathcal{M})}\right)\\
&\leq c_{abs}\left(\|A\|_{\mathcal{L}_{1}(\mathcal{M})}+\|M(A)\|_{\mathcal{L}_{1}(\mathcal{M})}+1\right)
\leq c_{abs}\left(1+\|\mu(A)\log_{+}(\mu(A))\|_{L_{1}(0,1)}\right) .
\end{split}\end{eqnarray*}

\end{proof}
\begin{rmk}
(i) Theorem \ref{L1 est} generalizes both the classical Zygmund's \cite[Corollary IV.6.9, p. 251]{BSh} and Riesz's results \cite[Corollary IV.6.10, p. 251]{BSh} (see also \cite[Theorem V.6.6, p. 248]{BSh}, \cite[Corollary IV.6.8, p. 251]{BSh}).

(ii) Let $\mathcal{H}$ be the non-commutative Hilbert transform associated with finite maximal subdiagonal algebras defined as in \cite{Ran1}. Then, $\mathcal{H}$ satisfies the assumptions of the Theorem \ref{first main
theorem} (see \cite[Theorem 2]{Ran1}). In particular, Theorem \ref{L1 est} implies the result of \cite[Theorem 4]{Ran1}.
\end{rmk}

\section{Acknowledgment}
The first and third authors were partially supported by Australian Research Council. The second author was partially supported by the grant of the Science Committee of the Ministry of Education and Science of the Republic of Kazakhstan. The authors thank the anonymous referee for useful comments which improved the exposition of the paper and for bringing reference \cite{HU}
to their attention.

\end{document}